\documentclass[11pt]{amsart}

\usepackage{amsmath,amssymb,amscd,fullpage,pb-diagram, mathdots}
\usepackage[pdfstartview=FitH,pdfmenubar=true]{hyperref}
\usepackage{bbm}                                   % provide support for numerical mathbb fonts     %
\usepackage{todonotes}
\usepackage{comment}

\theoremstyle{definition}

\newcommand{\Maa}{{\on{Maa}}}
\newcommand{\mdiag}[2]{\begin{pmatrix} #1 & \\ & #2 \end{pmatrix}}
\newcommand{\adiag}[2]{\begin{pmatrix} & #1 \\ #2 & \end{pmatrix}}
\newcommand{\unip}[1]{\begin{pmatrix} 1 & #1 \\ & 1 \end{pmatrix}}

\newtheorem{thm}[equation]{Theorem}
\newtheorem{cor}[equation]{Corollary}
\newtheorem{conj}[equation]{Conjecture}

\newtheorem{prop}[equation]{Proposition}
\newtheorem{lem}[equation]{Lemma}
\newtheorem{rmk}[equation]{Remark}

\newtheorem{defn}[equation]{Definition}

\numberwithin{equation}{subsection}

\newcommand{\bpm}{\begin{pmatrix}}
\newcommand{\epm}{\end{pmatrix}}
\newcommand{\bsm}{\begin{smallmatrix}}
\newcommand{\esm}{\end{smallmatrix}}
\newcommand{\bspm}{\left(\begin{smallmatrix}}
\newcommand{\espm}{\end{smallmatrix}\right)}
\newcommand{\bm}{\begin{matrix}}
\renewcommand{\em}{\end{matrix}}
\newcommand{\bbm}{\begin{bmatrix}}
\newcommand{\ebm}{\end{bmatrix}}

\newcommand{\bs}{\backslash}

\newcommand{\C}{\mathbb{C}}

\newcommand{\A}{\mathbb{A}}
\newcommand{\Q}{\mathbb{Q}}
\newcommand{\Z}{\mathbb{Z}}
\newcommand{\R}{\mathbb{R}}

\newcommand{\diag}{\operatorname{diag}}
\newcommand{\pr}{\operatorname{pr}}

\newcommand{\Span}{\operatorname{Span}}
\newcommand{\Arg}{\on{Arg}}

\newcommand{\la}{\langle}
\newcommand{\ra}{\rangle}

\newcommand{\vph}{\varphi}

\newcommand{\on}{\operatorname}

\renewcommand{\Re}{\on{Re}}
\renewcommand{\Im}{\on{Im}}
\newcommand{\gm}{\gamma}
\newcommand{\Gm}{\Gamma}
\newcommand{\sg}{\sigma}

\newcommand{\wt}{\widetilde}
\newcommand{\Mat}{\on{Mat}}
\newcommand{\f}{\mathfrak}
\renewcommand{\c}{\mathcal}

\newcommand{\ord}{\on{ord}}
\newcommand{\SL}{\wt{SL}}
\newcommand{\ad}{{\on{ad}}}

\newcommand{\hs}{\Big|_{\frac12}^{\on{Hol}}}
\newcommand{\ms}{\Big|_{\frac12}^{\on{Maa}}}

\newcommand{\sgn}{\operatorname{sgn}}
\newcommand{\1}{\mathbbm{1}}

\newcommand{\fin}{\mathrm{f}}

\begin{document}

\author{Joseph Hundley and Qiao Zhang}
\email{jhundley@math.siu.edu}
\address{Math. Department, Mailcode 4408,
Southern Illinois University Carbondale,
1245 Lincoln Drive Carbondale, IL 62901}
\email{q.zhang@tcu.edu}
\address{Department of Mathematics
Texas Christian University
Fort Worth, TX 76129}

\title{Fourier Coefficients of Theta Functions at Cusps other than Infinity}
\date{\today}
\maketitle
\tableofcontents

\section{Introduction}

In this paper, we use the adelic theory to study the Fourier coefficients of twisted theta functions at different cusps. The main results are Theorems~\ref{thm:MainFirst} and~\ref{thm: level higher}.

The theory of theta functions has a long history, going back
to Jacobi, and has a broad range of applications
throughout various branches of mathematics.
See, for example,~\cite{Mumford}.

In number theory, theta functions may be used to
study representation numbers of quadratic forms
(see~\cite[Chapter 11]{Iwaniec}), or
to shed light on the Shimura correspondence (\cite{Shimura})
between modular forms of integral weights and those of
half integral weights (see~\cite{Shintani},~\cite{Katok-Sarnak},~\cite{Waldspurger}).
Moreover, in 1976, Serre and Stark~\cite{Serre-Stark} answered a question
of Shimura, by showing that theta functions actually span the
space of all modular forms of weight $\frac 12$ for $\Gm_0(N)$.

An important development in the theory of theta functions
was their interpretation in terms of a representation
of the metaplectic group.  This point of view goes back
to~\cite{Weil}. It leads naturally to a vast generalization
of the Shimura correspondence, known as the theta
correspondence,
as well as a local analogue in the representation theory of
groups over local fields. (See the survey article~\cite{Prasad}, and its references.)
For $GL(2),$ the theory was significantly explicated by~\cite{Gelbart},
using explicit results of~\cite{Kubota}.
The results of Serre and Stark were explained from this
point of view by Gelbart and Piatetski-Shapiro~\cite{Gelbart-PS}.

The theory of Fourier expansions of automorphic forms
at various cusps goes back to~\cite{Roelcke} and~\cite{Maass}.
We briefly review the concept.
Our treatment follows~\cite{Iwaniec}.

To define Fourier coefficients at a cusp $\f a$ one must choose a suitable ``scaling'' matrix.
This notion is actually relative to a choice of discrete group
$\Gm.$  Assume for simplicity that $-I \in \Gm.$
Then a scaling matrix for $\f a$ relative to $\Gm$ is
a matrix
 $\sg\in SL(2, \R)$ which maps $\infty$ to $\f a$
 and conjugates the stabilizer of $\f a$ in $\Gm$ to the
 stabilizer of $\infty$ in $SL(2, \Z).$
 Having fixed such a matrix $\sg$
 to define Fourier coefficients at $a$ of a modular form
 $f$ for the group $\Gm,$
   one
 studies the function
 $f\hs \sg,$ where $\hs \sg$ is the weight one half slash operator corresponding to $\sg.$ Its definition is reviewed in
 Section \ref{ss:slash ops}.
In general, the function $f\hs \sg$ has a Fourier
expansion supported on some additive shift of the
integers. That is
$$f\hs \sg(z)
=
\sum_{n=0}^\infty
A_{f}(\sg, n+ \kappa_{f, \f a})
e^{2 \pi i (n+ \kappa_{f, \f a} )},
$$
for some constant $\kappa_ \f a\in \Q \cap [0,1)$ and coefficients $A_f ( \sg, n+\kappa_{f, \f a}), (n \ge 0).$
We refer to the coefficients $A_f ( \sg, n+\kappa_{f, \f a})$ appearing
in this expansion as the Fourier coefficients of $f$ at $\f a$, defined relative
to $\sg.$
Their dependence on the choice of $\sg$ is addressed by
lemma \ref{lem: dependence on choice of scaling matrix}.
Note that a modular form for $\Gm$ is also a modular form
for any subgroup of $\Gm.$ Replacing $\Gm$ by a subgroup
may alter which matrices $\sg$ are considered
suitable to be scaling matrices.

The constant $\kappa_{f, \f a}$ is called the cusp parameter
of $f$ at $\f a.$  It actually depends only on the multiplier
system of $f$ (defined as in ~\cite[\S 2.6, \S 2.7]{Iwaniec}).
As a byproduct of our computations, we show that
$\kappa_{f, \f a}$ is $0$ for all cusps $\f a$ and all of the
theta functions which we consider.
That is, the Fourier expansion of a theta function at
every cusp is supported on the integers. In fact, it is supported
on the squares.

This paper was prompted by some numerical computations
of Dorian Goldfeld and Paul Gunnells.
Let $\chi_2$ and $\chi_3$ be the unique nontrivial quadratic Dirichlet characters modulo $4$ and $3$ respectively, let
$\chi = \chi_2\chi_3$ (a Dirichlet character modulo $12$)
and let
$$\theta_\chi(z) = \sum_{n=1}^\infty \chi(n) e^{2\pi i n^2 z}.$$
Then $\theta_\chi$ is a modular form
of weight $1/2$ and level 576.
Goldfeld and Gunnells computed the Fourier coefficients of
$\theta_\chi$
 and discovered that (for suitable scaling matrices) the sequence of
its Fourier coefficients at every other
cusp was simply a scalar multiple of the sequence of
the Fourier coefficients at infinity.
For the cusps which can be transported
to $\infty$ by a Fricke involution, this is
expected, in view of the results of~\cite{Asai}.  (See also~\cite{Kojima}.)
For the other cusps, the result is more surprising.  But the phenomenon
also suggests an explanation: this theta function must correspond
to an element of the Weil representation which is fixed, up to scalars,
by a group which is larger than $\Gm_0(576)$, and acts transitively on
the cusps. In Theorem~\ref{thm:MainFirst} we prove this.

A remark is in order. Recall that for integral $k,$ the
weight $k$
slash operators define a right action of $SL(2, \R)$
on the vector space of modular forms of weight $k.$
For half-integral weight modular forms, this is not
the case.  Rather, we have
\[
\hs \sg_1 \hs \sg_2 = \pm \hs (\sg_1 \sg_2),
\]
with the occasional minus sign resulting from branch cuts
in the square root. Thus the ``larger group'' that we
define is actually a subgroup of a covering group.

The result can be recast classically as follows.
Let $$
\Gm^{(24)}
=\bpm 1 & \\ & 24 \epm SL(2, \Z) \bpm 1& \\ & 24^{-1} \epm.
$$
Then $\Gm^{(24)}$ acts transitively on the cusps and contains
a scaling matrix for each of them relative to
$\Gm_0(576).$ Moreover, there is a function
$\zeta$ mapping $\Gm^{(24)}$ into the group of $24$th roots
of $1$ such that $\theta_\chi \hs \sg = \zeta(\sg) \theta_\chi$
for each $\sg \in \Gm^{(24)}.$
The function $\zeta$ is not a homomorphism, but it does
satisfy $\zeta(\sg_1) \zeta(\sg_2) = \pm \zeta(\sg_1\sg_2).$

Theorem~\ref{thm: level higher} extends this
to the twisted theta functions of higher levels.
Here again, we were led by numerical computations and a
conjecture of Goldfeld and Gunnells.  We first state the
Goldfeld-Gunnells conjecture for the special case we refer to
as the five twist.
Let  $\chi_5$ be the unique primitive quadratic character modulo $5.$
Set
$$\theta_{\chi_5} (z) =
\sum_{n=1}^\infty
\chi_5(n)\chi(n)e^{2\pi in^2z}.$$
Then the conjecture of Goldfeld and Gunnells, motivated by
their numerical computations, is as follows.
\begin{conj}[Goldfeld-Gunnells]
Let
$$
a = \frac{2 \sin( \frac{4\pi}5)}{\sqrt{5}}= \frac{\sqrt{(10 - 2 \sqrt 5)/5}}2 \approx 0.52573,
\qquad
b = \frac{2 \sin( \frac{2\pi}5)}{\sqrt{5}}=\frac{\sqrt{(10 + 2 \sqrt 5)/5}}2 \approx 0.85065.
$$
Let $\f a = u/w \in \Q,$
and let $\sg_\f a \in SL(2, \R)$ be any scaling matrix
for $\f a.$
  If $5 \nmid w$ or $25 \mid w$ then
$$
|A_{\theta_{\chi_5}}(\sg_\f a, n^2) |
= \begin{cases}
1, & \text{ if } 5 \nmid n, \\ 0, & \text{ if } 5 \mid n.
\end{cases}
$$
On the other hand, if $5 || w$ then either
$$
|A_{\theta_{\chi_5}}(\sg_\f a, n^2) | = \begin{cases} a, & \text{ if }5 \nmid n , \\
2b, & \text{ if } 5 \mid n,
\end{cases},
$$
or
$$
|A_{\theta_{\chi_5}}(\sg_\f a, n^2) | = \begin{cases} b, & \text{ if }5 \nmid n , \\
2a, & \text{ if } 5 \mid n,
\end{cases}.
$$
\end{conj}
\begin{rmk}
The fact that $|A_{\theta_{\chi_5}}(\sg_\f a, n^2) |$
is independent of the choice of $\sg_\f a$
is an easy consequence of  Lemma \ref{lem: dependence on choice of scaling matrix}.
\end{rmk}
This conjecture was generalized to several larger
primes $p$ by Gunnells, who replaced
 $\chi_5$ by a Dirichlet
character $\chi_p$ modulo $p$ which factors through the squaring
map $(\Z/p\Z)^\times \to  (\Z/p\Z)^\times.$
For  $\f a = u/w,$ with $p\nmid w$ or $p^2 \mid w,$
the extension is direct. For $p||w,$ Gunnells
predicts that $|A_{\theta_{\chi_p}}(\sg_\f a, n^2)|$ depends only on the
image of $n^2$ in $\Z/p\Z,$ so that the full sequence
$(|A_{\theta_{\chi_p}}(\sg_\f a,n^2)|)_{n=1}^{\infty}$ is determined by
a subsequence of length $(p+1)/2.$ Moreover, as $\f a$
ranges over cusps, only $p-1$ distinct sequences should appear, each consisting of zeros of explicit integer polynomials.
These $p-1$ sequences come in two classes. For each class
there is an element corresponding to zero, and then a cycle
of $(p-1)/2$ other values. As $n$ runs through the nonzero
squares modulo $p,$ the sequence $ (|A_{\theta_{\chi_p}}(\sg_\f a,n^2)|)$
runs through this cycle. It may start at any point in the cycle and
this accounts for the total of $p-1$ possibilities.

For specific $p$ and $\f a,$ the  Goldfeld-Gunnells
and Gunnells
conjectures can be checked using
 theorem~\ref{thm: level higher}.
In this theorem, we
again produce a larger group which acts transitively on the cusps.
It no longer fixes the one-dimensional space spanned by our element of the Weil representation.
Instead, it fixes a finite dimensional space containing it,
and this permits us to obtain explicit results concerning
the Fourier coefficients at all the cusps.

\medskip

As an example, we consider the case $p=5.$ In this case,
the group we consider is
$$
\Gm^{(120)}
=\bpm 1 & \\ & 120 \epm SL(2, \Z) \bpm 1& \\ & 120^{-1} \epm.
$$
As a conjugate of $SL(2, \Z)$ this group certainly acts
transitively on the cusps, and it is easily verified that it
provides a scaling matrix for every cusp relative to
 $\Gm_0(14400).$
Now, let $V$ be the three dimensional complex vector space spanned by $\theta_\chi, \theta_{\chi_5}$ and $\theta_{\chi}^{(5)},$ where
$$
\theta_{\chi}^{(5)}( z) = \sum_{{n=1}}^\infty
\chi(5n) e^{2\pi i (5n)^2 z}.
$$
Then, Theorem \ref{thm: level higher} gives the following.
\begin{thm}
For each $\sg \in \Gm^{(120)},$ there
is a matrix $M(\sg^{-1}) \in GL_3(\C)$ such that
\begin{equation}\label{def of M(sig)}
\bbm\theta_{\chi_5} \hs\sg &
\theta_{\chi} \hs\sg  &\theta_{\chi}^{(5)} \hs\sg \ebm= \bbm  \theta_{\chi_5} & \theta_{\chi} &\theta_{\chi}^{(5)}  \ebm \cdot M(\sg^{-1}).
\end{equation}
\end{thm}

Like $\zeta,$ the function ${M:\Gm^{(120)} \to GL_3(\C)}$
defined implicitly by \eqref{def of M(sig)} is {\it not} a homomorphism,
but satisfies $M(\sg_1)M(\sg_2) = \pm M(\sg_1\sg_2).$
It is perhaps better understood by working with a
metaplectic covering groups.
This point of view is presented in the body of the paper.

Now  the Fourier
coefficients $A_{\theta_{\chi_5}}(\sg, m)$
 may be recovered from
\eqref{def of M(sig)}, provided one can compute $M(\sg^{-1})$
explicitly.
Specifically, if $^t \!\bbm c_1 & c_2 & c_3 \ebm$ is the first
column of $M(\sg^{-1})$ then
$A_{\theta_{\chi_5}}(\sg, m)=0$ for $m$ not a square, and
\begin{equation}
\label{c a's in terms of c i's}
A_{\theta_{\chi_5}}(\sg, n^2) = \chi(n)\cdot\begin{cases}
\chi_5(n) c_1 + c_2, & 5 \nmid n, \\
c_2+c_3, & 5 \mid n.
\end{cases}
\end{equation}

In order to compute $M(\sg)$ explicitly, it is helpful to work
one prime at a time.
If $p$ is
a prime, let
$$K_p^{(120)} = \bpm 1 & \\ & 120 \epm SL(2, \Z_p) \bpm 1& \\ & 120^{-1} \epm.$$
Clearly, $\Gm^{(120)} < K_p^{(120)}$ for each $p.$

We first describe a function $M_5: K_5^{(120)}\to GL_3(\C).$ Then $M$ will be the product of $M_5,$ similar but simpler contributions from $2$ and $3,$
and some additional factors. In defining and computing
$M_p,$ for $p=2,3,5,$ we make use of the fact (see Lemma \ref{lem: generators for K p}) that $SL(2, \Z_p)$ is generated
by the elements
\[
\left\{\bpm & 1 \\ -1 & \epm\right\}\bigcup\left\{ \bpm 1 & a \\ 0 & 1 \epm \mathrel{\left|\vphantom{\bpm 1 & a \\ 0 & 1 \epm}\right.} a \in \Z_p\right\}.
\]
Hence $K_p^{(120)}$ is generated by the conjugates of these
elements.

To describe
$M_5:K_5^{(120)} \to GL_3(\C).$
It is convenient to introduce the notation
$$
\cos_5(x) = \cos( - \{ x\}_5), \qquad \sin_5(x) = \sin(-\{ x_5\}).
$$
so that
\[
\cos_5(2\pi x) = \frac12 (e_5(x) + e_5(-x)), \qquad \sin_5(x) = -\frac1{2i} ( e_5(x) - e_5(-x)),
\]
where $\{\ \}_5$ and $e_5$ are defined as in
 section \ref{ss: local weil rep}.

Then
$$
M_5
\bpm 1& \frac a5\\ & 1\epm= \bpm
\cos_5(\frac{2\pi a}5) & -i\sin_5(\frac{2\pi a}5)&0 \\
-i\sin_5(\frac{2\pi a}5)&\cos_5(\frac{2\pi a}5) &0\\
i\sin_5(\frac{2\pi a}5)&1-\cos_5(\frac{2\pi a}5) &1
\epm
, \qquad ( a \in \Z_5),
$$
$$
M_5
\bpm a&\\&a^{-1}\epm=
\bpm \chi_5(a)&& \\
&1&\\&&1
\epm
, \qquad ( a \in \Z_5^\times)
$$
$$
M_5
\bpm &1/5\\ -5 & \epm =
\bpm -1&&\\&& \frac{1}{\sqrt{5}}\\ & \sqrt{5} & \epm
$$
$$
M_5(\gm_1 \gm_2) = \beta_5( \gm_1, \gm_2)M_5(\gm_1) M_5(\gm_2)
, \qquad (\gm_1, \gm_2 \in K_5^{(120)}),
$$
where
$\beta_v: SL(2, \Q_v) \times SL(2, \Q_v) \to  \{ \pm 1\}$ is defined for each place $v$ as in
section \ref{subsection: local metaplectic groups}.

For $p=2,$ we define a
function $M_2: K_2^{(120)} \mu_8$ (the eighth
roots of unity) by
$M_2(\sg) = \xi_2( \sg, 1),$ where
$\xi_2$ is defined as in theorem \ref{thm: k 2 8 acts on ph2 chi 2 by xi 2}.  The function $M_2$
satisfies $M_2( \sg_1) M_2( \sg_2) = M_2( \sg_1\sg_2)
\beta_2( \sg_1, \sg_2).$

Similarly, $M_3: K_3^{(120)} \to \mu_6$
is defined by $M_3(\sg) = \xi_3(\sg, 1)$ with
$\xi_3$ as in Theorem \ref{K 3 3 acts on phi chi 3 by xi 3},
and
satisfies
$$
M_3(\gm_1 \gm_2) = M_3(\gm_1) M_3(\gm_2)
\beta_3( \gm_1, \gm_2), \qquad (\gm_1, \gm_2 \in K_3^{(120)}).
$$

\begin{thm}
For any $\sg \in \Gm^{(120)},$ the equation \eqref{def of M(sig)} holds with
$$M(\sg) = s_\A(\sg)\beta_\infty( \sg^{-1}, \sg)M_2(\sg) M_3(\sg) M_5(\sg),$$
where $s_\A(\sg) \in \{ \pm 1\}$ is defined in
section \ref{ss: Global Metaplectic Group}.
\end{thm}

We now describe briefly how these results may be used to
verify the Goldfeld-Gunnells conjecture regarding
$\theta_{\chi_5}.$  First, since we are only interested in
the {\it absolute values} of the Fourier coefficients, we
 may replace $M(\sg^{-1})$ by  $M_5(\sg^{-1})$
 in \eqref{def of M(sig)}. As a scaling matrix for
the cusp $u/w$, where $u,w \in \Z, \gcd(u,w) =1$
and $w \mid 14400,$ we choose the matrix
$$
\sg:=\bpm \frac{u[w,120]}w & \frac r{120} \\ [w, 120] & s\epm$$
where $r, s \in \Z$ satisfy $120 us - rw = \gcd(w,120).$
First suppose that $25 \nmid w.$ Then we may write
$$
\sg^{-1} = \unip{-\frac{s}{[120,w]}}\adiag{\frac{1}{5}}{-5}\mdiag{[24,w_1]}{\frac1{[24,w_1]}}\unip{\frac{u}{w}},
$$
where $w_1 = w/(w,5),$
and we may use
\begin{equation}\label{product}
M\unip{-\frac{s}{[120,w]}}
M\adiag{\frac{1}{5}}{-5}M\mdiag{[24,w_1]}{\frac1{[24,w_1]}}M\unip{\frac{u}{w}},
\end{equation}
in lieu of $M(\sg^{-1})$ since the two are equal up to
a sign.
By a computation which is essentially a local analogue of Lemma \ref{lem: dependence on choice of scaling matrix},
the matrix $M\bspm 1& -s/[120,w]\\ 0 & 1 \espm$ multiplies
each coefficient by a root of unity without changing the absolute
value, so it may be omitted.
The product of the remaining terms in \eqref{product} is
$$
\bpm
- \chi_5( [24, w_1]) \cos_5( 2\pi u/w) &
- \chi_5( [24, w_1]) i\sin_5( 2\pi u/w) & 0\\
\frac{i}{\sqrt{5}}   \sin_5( 2\pi u/w) & \frac{1}{\sqrt{5}} ( 1 - \cos_5( 2\pi  u/w) ) & \frac1{\sqrt{5}} \\
\sqrt{5} i \sin_5( 2\pi u/w) & \sqrt{5}\cos_5 ( 2\pi u/w) & 0
\epm.
$$
Let $^t\bbm c_1 & c_2 & c_3\ebm$ be the first column
of this matrix.  Then   equation
\eqref{c a's in terms of c i's} holds up to roots of unity.
In the case $5 \nmid w,$ we have $c_2=c_3=0$ and
$c_1$ is a root of unity, yielding the Goldfeld-Gunnells conjecture in this case.
The case $5||w$ reduces to checking that
\[
|\pm \cos( 2\pi /5) \pm i \sin( 2\pi /5)/\sqrt{5}| = 2 \sin( 4\pi /5), \qquad |\pm \cos( 4\pi /5) \pm i \sin( 4\pi /5)/\sqrt{5}| = 2 \sin( 2\pi /5)
\]
which is straightforward.

In the case when $25||w,$ the scaling matrix is of the form
$$
\bpm a& b \\ 25c & d \epm
= \bpm 1 & b/d \\ 0 & 1 \epm
\bpm - d^{-1} & 0 \\ 0 & -d \epm
\bpm 0 & 1/5 \\ -5 & 0 \epm
\bpm 1 & c/d \\ 0 & 1 \epm
\bpm 0 & 1/5 \\ -5 & 0 \epm,
$$
with $5 \nmid d.$
But then $M_5\bspm 1 & c/d \\ 0 & 1 \espm $ and
$M_5\bspm 1 & c/d \\ 0 & 1 \espm $
are trivial, so up to a root of unity the last three terms simply
cancel, leaving only $M_5\bspm - d^{-1} & 0 \\ 0 & -d \espm,$
and the Goldfeld-Gunnells conjecture follows in this case as
well.

\medskip

The organization of the paper is as follows.  In Section 2 we review the classical theory of the theta functions which
we shall study.  In Sections 3 and 4 we develop the relevant theory of local metaplectic groups and Weil representations, including the explicit formulae which are crucial for our aims in this paper.  In Section 5 we
review the relevant notions regarding the global metaplectic group, In Section 6 we define adelic theta functions corresponding to the classical objects reviewed in Section 2. The main theorems are proved in Section 7.

It may be noted that in this paper we have restricted attention to theta series attached to Dirichlet characters which are even at every prime except for 2 and 3. However, it seems that the method extends to other characters in a natural way. We hope to return to this in future work.

We would like to thank Dorian Goldfeld and Paul Gunnells for
stimulating this research, and for sharing the results of their
computations, which provided a perfect (and much needed!) method of checking the formulae
which came out of our work.
The work was undertaken during a special semester
at ICERM and we would like to thank ICERM for
providing a fantastic working environment.
JH was supported by NSF Grant
DMS-1001792
and gratefully thanks the NSF for the support.

\section{The Classical Theory}

\subsection{The Scaling Matrices}

Let $\c H$ denote the upper half plane. We shall make use of the classical action
of $SL(2, \R)$ on $\c H$ by fractional linear transformations:
$$
\bpm a&b\\c&d \epm \cdot z = \frac{az+b}{cz+d},
\qquad \left( \bpm a&b\\c&d \epm\in SL(2, \R), \;\; z \in \c H\right).
$$

Let $\f a$ be a cusp for $\Gamma_0(N)$.
Write $\Gm_\f a$ for the stabilizer of $\f a$ in $\Gm_0(N).$
Choose $\gm \in SL(2, \Z)$ with $\gm \infty = \f a.$
Then
\[
\gm^{-1} \Gm_\f a\gm\subset \Gm_\infty
= \la \bspm -1&\\ & -1 \espm ,
\bspm 1& 1\\ &1 \espm \ra.
\]
Moreover $\gm^{-1} \Gm_\f a\gm$ contains $\bspm -1&\\ & -1 \espm.$
It follows that
\[
\gm^{-1} \Gm_\f a\gm
= \la \bspm -1&\\ & -1 \espm ,
\bspm 1& m_\f a\\ &1 \espm \ra
\]
for a unique positive
integer $m_\f a$ called the width of $\f a$ (relative to $\Gm_0(N)$).
Note that the matrix
\[
g_\f a=\gm \bspm 1& m_\f a\\ &1 \espm \gm^{-1}
\]
is independent of the choice of $\gm$, as the elements
$\bspm 1& m_\f a\\ &1 \espm$ and $\bspm 1& -m_\f a\\ &1 \espm$
are not conjugate in $SL(2, \R).$ Hence we have $\Gm_\f a = \la -I, g_\f a\ra.$

Now choose $\sg\in SL(2, \R)$ such that
\[
\sg \cdot \infty = \f a \qquad \sg \bspm 1& 1 \\ &1 \espm \sg^{-1}=g_\f a.
\]
Following~\cite{Iwaniec}, we refer to $\sg$ as a ``scaling
matrix'' for $\f a.$ Clearly, $\sg$ is unique up to an element of the
centralizer of $\bspm 1&1\\0&1 \espm$ in $SL(2, \R),$
which is the subgroup
\[
\{ \bspm \varepsilon & \varepsilon t \\ & \varepsilon \espm : \varepsilon \in\{ \pm 1\},
t \in \R\}.
\]

If $\f a = \infty$ then $m_\f a= 1,$ $g_\f a = \bspm 1&1\\&1 \espm,$
and one can take $\sg$ to be the identity.

If $\f a = \frac uw \in \Q,$ then
\[
m_\f a = \frac{N}{(w^2, N)}, \qquad g_\f a = I + \frac N{(w^2, N)} \cdot \bspm -uw & u^2 \\ -w^2 & uw \espm.
\]
As for the scaling matrix, a common choice is
\[
\sg^0 = \bpm
\frac uw \sqrt{[w^2, N]}& \\ \sqrt{[w^2, N]}
& \frac wu \sqrt{[w^2, N]}^{-1}
\epm,
\]
where $[m,n]$ is the least common multiple of $m$ and $n.$
If $N=M^2$, this simplifies to
\begin{equation}\label{eq:ScalingMatrixCommon}
\sg^0 = \bpm
\frac uw [w, M]& \\ [w, M]
& \frac wu [w, M]^{-1}
\epm.
\end{equation}
This choice, however, is not suitable for our purpose. Instead, we would like to have our scaling matrix in a particular conjugate of $SL(2,\Z)$.

For any integer $M$, let
$$
\Gm^{(M)}:= \bpm 1 & \\ & M \epm SL(2, \Z) \bpm 1& \\ & M^{-1} \epm.
$$

\begin{lem}\label{lem:ScalingMatrixConstruction}
Let $M\geq 1$, and let $\f a=\frac uw\in \Q$ be a cusp for $\Gamma_0(M^2)$. Then there exists a scaling matrix
$\sg$ of $\f a$ that lies in $\Gm^{(M)}.$ Explicitly, choose $r', s' \in \Z$ with
\[
uMs'-r'w = (M,w),
\]
then we may take the scaling matrix
\begin{equation}\label{eq:ScalingMatrixSpecial}
\sg=\bpm 1&\\ & M\epm
\bpm \frac{uM}{(M,w)} & r' \\ \frac{w}{(M,w)} & s' \epm
\bpm 1&\\ & \frac1{M}\epm
= \bpm \frac{[M,w]u}w & \frac {r'}{M}\\ [M, w] & s'
\epm\in\Gm^{(M)}.
\end{equation}
\end{lem}

\begin{proof}
This follows from the direct computation.
\end{proof}

\begin{rmk}
To illustrate the relation between the usual choice of scaling matrices and our choice, we have
$$
\sg = \sg^0\bpm 1& \frac{wr'}{M u [M,w]} \\ & 1 \epm.
$$
In particular, the condition $\sigma\in\Gamma^{(M)}$ determine $r'$ uniquely modulo $uM/\gcd(w,M)$
so the quantity $\frac{wr'}{M u [M,w]}$ is determined modulo $1/M.$
\end{rmk}

\begin{lem}
Under the notations in Lemma~\ref{lem:ScalingMatrixConstruction}, we have the decompositions
\begin{equation}\label{decomp for case vp(w) le vp(M)}
\sigma^{-1}
= \bpm 1& \frac{s'}{[w,M]}\\ 0 & 1 \epm
\bpm 0 & \frac 1{[w,M]}\\ -[w,M] & 0 \epm
\bpm 1& -\frac uw \\ & 1 \epm ,
\end{equation}
\begin{equation}\label{decomp for case vp(w) > vp(M)}
\sigma^{-1}
= \bpm 1 & -\frac{r' w}{[w,M]Mu} \\ & 1 \epm
\bpm \frac{ -w}{[w,M]u} & \\ & -\frac{u[w,M]}{w} \epm
\bpm 0& M^{-1} \\ -M& 0 \epm
\bpm 1& \frac{w}{uM^2}\\ & 1\epm
\bpm 0& M^{-1} \\ -M& 0 \epm .
\end{equation}
\end{lem}
\begin{proof}
Direct calculation.
\end{proof}

\begin{rmk}
Let $p|M$ and
$$
K_p^{(M)}:= \bpm 1 & \\ & M \epm SL(2, \Z_p) \bpm 1& \\ & M^{-1} \epm.
$$
The merit of the above lemma is that it gives an explicit decomposition of $\sigma^{-1}$ within the group $K_p^{(M)}$. More precisely, if $v_p(w) \le v_p(M)$, then
each matrix in \eqref{decomp for case vp(w) le vp(M)} is
an element of $K_p^{(M)},$ while if $v_p(w)=2$ while  $v_p(M)=1,$
each matrix in \eqref{decomp for case vp(w) > vp(M)} is an element of $K_p^{(M)}.$
\end{rmk}

\subsection{The Slash Operators}\label{ss:slash ops}

We define
$$
j: SL(2, \R) \times \c H \to \c H,
\qquad
j\left( \bpm a&b\\c&d \epm ,\;\; z \right)
= cz+d, \qquad \left( \bpm a&b\\c&d \epm\in SL(2, \R), \;\; z \in \c H\right).
$$
It satisfies the cocycle condition
$$
j(\gm_1\gm_2,z) = j( \gm_1, \gm_2\cdot z)j( \gm_2, z).
$$
We also define $\f j( \gm , z) = j(\gm, z)/|j( \gm , z)|.$
Observe that
$$
\f j \left(\bpm y^\frac 12 & xy^{-\frac12}\\ & y^{-\frac12}\epm
\bpm \cos \theta & - \sin \theta \\ \sin \theta & \cos \theta
\epm
, i
\right)
=\f j \left(
\bpm \cos \theta & - \sin \theta \\ \sin \theta & \cos \theta
\epm
, i
\right)
=e^{i\theta},
$$
for all $x,y, \theta \in \R$ with $y >0.$

For $g \in SL(2, \R)$ and $f: \c H \to \C$
define
$$
\left(f \hs g\right)(z) = j(g, z)^{-\frac 12} f(g \cdot z),
$$
$$
\left(f \ms g\right)(z) = \f j(g, z)^{-\frac 12} f(g \cdot z).
$$
Here, the square roots are defined to be the principal
value, having an argument in $(-\pi/2, \pi/2).$
If $\mu(f)$ is the function defined by
$[\mu(f)](z) = f(z) \Im(z)^{\frac 14},$ then for every $g\in SL(2,\R)$ we have
$$
\mu\left(f \hs g\right)=
\mu(f) \ms g.
$$

Observe that although $j$ and $\f j$ are cocycles,
their square roots are not, because of the discontinuity of the
principal value.

\subsection{Modular Forms of Weight $1/2$}

Let $f$ be a modular form of weight $\frac 12$
and multiplier $\vartheta$
for $\Gm_0(N),$ as in~\cite[\S 2.6, \S 2.7]{Iwaniec}.
Thus
$$
f\big|^{\on{Hol}}_{\frac12} \gm = \vartheta( \gm)\chi_c(d)\varepsilon_d f ,
 \qquad ( \forall \gm=\left(\begin{smallmatrix} a & b \\ c & d \end{smallmatrix}\right) \in \Gm_0(N)).
$$
At every cusp $\f a$, let $\kappa_{f, \f a}$ be the unique (rational) solution
to $e^{2\pi i \kappa_{\f a}} = \vartheta(g_\f a)$ in $[0,1).$ Then the function $e^{-2\pi i \kappa_{\f a} z} f\big|^{\on{Hol}}_{\frac12}\sg$
is periodic of period $1$ (see~\cite[Page 43]{Iwaniec}),
and the Fourier
coefficients of $f$ at $\f a$ relative to $\sg$ are the coefficients
$A_f( \sg, n+\kappa_{f, \f a})$
in its Fourier expansion:
$$
\left(f\hs \sg\right)(z) = \sum_{n =0}^\infty A_f(\sg, n+\kappa_{f, \f a}) e^{2\pi i (n+\kappa_{f, \f a})z}.
$$

\begin{lem}\label{lem: dependence on choice of scaling matrix}
We have
$$
A_f\left(\sg\bpm \varepsilon & \\ & \varepsilon \epm,\;\;  n+\kappa_{f, \f a}\right)
=
A_f\left(\sg,\;\;  n+\kappa_{f, \f a}\right)\vartheta\bpm \varepsilon &  \\ & \varepsilon \epm
$$
$$
A_f\left(\sg \bpm 1 &  t \\ & 1 \epm, \;\; n + \kappa_{f, \f a}\right)
=A_f\left(\sg, \;\; n + \kappa_{f, \f a}\right)e^{2\pi i(n+\kappa_{f, \f a})t}.
$$
\end{lem}

\begin{proof}
This is obvious; for example, the latter comes immediately from
$$\begin{aligned}
\sum_{n=0}^\infty
A_f\left(\sg \bpm 1 &  t \\ & 1 \epm, \;\; n + \kappa_{f, \f a}\right)
e^{2 \pi i(n+\kappa_{f, \f a})z}
=
\left(f\hs  \sg \bpm 1 &  t \\ & 1 \epm\right)(z)\\
=
\left(f\hs  \sg\right)(z+t)
= \sum_{n=0}^\infty
A_f\left(\sg, \;\; n + \kappa_{f, \f a}\right)
e^{2\pi i (n+\kappa_{f, \f a})(z+t)}. \qedhere
\end{aligned}
$$
\end{proof}
\subsection{The Classical Theta Functions}

Let $\chi\pmod{N}$ be an even Dirichlet character,
%
%,
%
%\[
%\chi(n)=\left(\frac{12}{n}\right)=\begin{cases}
% 1, & \text{if $n = 1,11 \pmod{12}$;} \\
%-1, & \text{if $n = 5,7 \pmod{12}$;} \\
% 0, & \text{otherwise.}
%\end{cases}
%\]
then the classical theta function
\[
\theta_\chi(z) = \sum_{n=1}^\infty \chi(n) e^{2\pi i n^2z}, \qquad
(z \in \c H)
\]
(Cf.~\cite[\S 10.5]{Iwaniec}) is a cusp form of weight $1/2$ and level $4N^2,$ and we have
\[
\theta_\chi(\gamma z)=\chi(d)\chi_c(d)\varepsilon_d^{-1}(cz+d)^{1/2}\theta_\chi(z), \qquad \left(\gamma=\begin{pmatrix} a & b \\ c & d \end{pmatrix}\in\Gamma_0(4N^2)\right),
\]
where
\[
\varepsilon_d=\begin{cases}
1, & \text{if $d\equiv 1\pmod{4}$;} \\
i, & \text{if $d\equiv 3\pmod{4}$,}
\end{cases}
\]
$\chi_t$ denotes the primitive character corresponding to the field extension $\Q(\sqrt{t})/\Q$ \footnote{For example, if $t$ is a perfect square, then $\chi_t=1$; if $t$ is not a perfect square, and let $D$ be the discriminant of $\Q(\sqrt{t})/\Q$, then $\chi_t$ is the primitive quadratic character of conductor $|D|$ given by
\[
\chi_t=\left(\frac{|D|}{\cdot}\right).
\]}, and the square root takes the principal value. In other words, we have
\[
\left(\theta_\chi \hs \gamma\right)(z)=\chi(d)\chi_c(d)
\varepsilon_d^{-1}
\theta_\chi(z) \qquad (\gamma\in\Gamma_0(4N^2)).
\]

In this paper, we consider the special cases that $N=12$ or $N=12p$ for some prime $p$, and study its Fourier coefficients at different cusps. The main results are Theorems~\ref{thm:MainFirst} and~\ref{thm: level higher}.

\section{Local Metaplectic Groups and Weil Representations}

\subsection{Local Metaplectic Groups}
\label{subsection: local metaplectic groups}
Let $\Q_v$ be one of the completions of $\Q.$
Thus $\Q_v= \Q_\infty = \R$ or $\Q_v$ is the p-adic
numbers $\Q_p$ for some prime $p.$
The Hilbert symbol on $\Q_v$ will be denoted by
$(\ , \ )_v,$ or, when $v$ is fixed, by $(\ , \ ).$
As in~\cite{Gelbart}, we define the cocycle
\begin{multline*}
\beta_v: SL(2, \Q_v) \times SL(2, \Q_v) \to \{ \pm 1\}, \\
(g_1,g_2)\mapsto\Big(x(g_1),x(g_2)\Big)_v\Big(-x(g_1)x(g_2),x(g_1g_2)\Big)_vs(g_1)s(g_2)s(g_1g_2),
\end{multline*}
where
$$
x\bpm a&b \\ c&d \epm
= \begin{cases}
c, & \text{if $c\neq 0$;} \\
d, & \text{if $c=0$,}
\end{cases} \qquad
s\bpm a&b \\ c&d \epm
=\begin{cases}
(c,d)_v, & \text{if $v<\infty$, $cd\neq 0$ and $\ord(c)$ is odd;} \\
1,       & \text{if otherwise.}
\end{cases}
$$
In particular, in the Borel subgroup it simplifies to
$$
\beta_v \left(
\bpm a_1& b_1 \\ & d_1 \epm , \quad
\bpm a_2& b_2 \\ & d_2 \epm
\right)=
(a_1, d_2)_v,
$$
and over $SL(2, \R)$ it simplifies to
\begin{equation}\label{beta for SL(2,R)}
\beta_\infty(g_1, g_2)=\Big(x(g_1),\;x(g_2)\Big)_{\infty}\Big(-x(g_1)x(g_2), \;x(g_1g_2)\Big)_{\infty}.
\end{equation}

We may then define a double cover of $SL(2, \Q_v),$
denoted $\SL(2, \Q_v)$ and consisting of the
set $SL(2, \Q_v) \times \{ \pm 1\}$
equipped with the operation
$$
( g_1, \zeta_1)(g_2 ,\zeta_2) := \Big(g_1g_2, \beta_v(g_1, g_2)\zeta_1\zeta_2\Big),
\qquad \Big(g_1, g_2 \in SL(2, \Q_v), \; \;
\zeta_1, \zeta_2 \in \{ \pm 1\}\Big).
$$
The function
$\pr: \SL(2, \Q_v) \to SL(2, \Q_v )$ given by $\pr(g, \zeta) =g$
is a homomorphism.

\subsection{Generators for $SL(2, \Z_p)$ and its preimage}

\begin{lem}\label{lem: generators for K p}
The group $K_p=SL(2, \Z_p)$ is generated by
\begin{equation}\label{generators}
\left\{\begin{pmatrix} & 1 \\ -1 & \end{pmatrix}\right\}\bigcup\left\{\begin{pmatrix} 1 & x \\ & 1 \end{pmatrix} \mid x\in\Z_p\right\}
\end{equation}
\end{lem}
\begin{proof}
Write $H_p$ for the subgroup of $SL(2, \Z_p)$ generated by
\eqref{generators}.
Then $H_p$ contains $\bspm 1&\\ x&1 \espm$ for each $x.$
Since
$$
\bpm 1& x-1\\ &1 \epm
\bpm 1&\\ 1&1\epm
\bpm 1& -\frac{x-1}{x} \\ &1 \epm
\bpm 1&\\ -x& 1 \epm
= \bpm x &\\& x^{-1} \epm,
$$
for all $x \in \Z_p^\times,$ it follows that
$H_p$ contains all diagonal, and hence all upper or lower triangular
elements of $K_p.$
Now consider an arbitrary element $\bspm a&b \\ c& d \espm$
of $K_p.$  Either $c$ or $d$ is a unit.
If $d$ is a unit then $-b/d \in \Z_p$ and
$\bspm 1& -b/d \\ & 1\espm
= \bspm a&b \\ c& d \espm$
is upper triangular, and hence in $H_p,$ which then
implies that $\bspm a&b\\ c& d \espm \in H_p.$
If $c$ is a unit then $a/c$ and $d/c$ are in $\Z_p$ and
$$
\bpm 1& -\frac ac\\ & 1 \epm
\bpm a& b \\ c& d \epm
\bpm 1& -\frac dc \\ & 1 \epm =\bpm &c^{-1}\\ c&\epm \in H_p,
$$
which completes the proof.
\end{proof}

\subsection{The Real Metaplectic Group}

In our study of the real metaplectic group $\SL(2,\R)$, the basic tools are the metaplectic analogue of the classical Iwasawa decomposition and the corresponding slash operator. Recall that
\[
SL(2,\R)=B^{+}(\R)\times SO_2(\R),
\]
where
$$B^+(\R) = \left\{
\bpm y^{\frac12}& xy^{-\frac12} \\ & y^{-\frac12} \epm
: y \in (0,\infty), \; x \in \R\right\}\le SL(2, \R).$$
Hence our discussion starts with the metaplectic preimage of $SO_2(\R)$.

For $\theta \in \R$ define
$$
\kappa_\theta = \bpm \cos \theta & -\sin \theta \\ \sin \theta & \cos \theta \epm \in SO_2(\R) \subset SL(2, \R),
$$
and
$$
\wt \kappa_\theta = \Big( \kappa_{2\theta}, \zeta(\theta)\Big),
$$
where $\zeta$ is the unique function $\R \to \pm 1$
which is periodic mod $2\pi$ and satisfies
$$\zeta( \theta) =
\begin{cases}
1,  & \text{if $-\pi/2 \le \theta < \pi/2$;} \\
-1, & \text{if $\pi/2 \le \theta < 3\pi/2$.}
\end{cases}
$$
\begin{lem}
The function $\theta \mapsto \wt \kappa_\theta$
is a homomorphism.
\end{lem}
\begin{proof}
Using  \eqref{beta for SL(2,R)}, one checks that
$$
\beta_{\infty}\left(
\kappa_{2\theta_1}, \kappa_{2\theta_2}
\right)= -1
\iff \frac{\zeta(\theta_1+\theta_2)}{\zeta(\theta_1)\zeta(\theta_2)}=-1
$$
on a case-by-case basis.
\end{proof}
We denote the set of the images $\wt \kappa_{\theta}$ by $\wt K.$
\begin{rmk}\label{rmk on sqrt'}
If we define $\sqrt{z}'$ for $z \in \C^\times$
so that $\Arg(\sqrt{z}') \in
%(
[-\frac \pi2, \frac\pi2)
%]
$ for all $z \in \C^\times,$
then  the function
$S^1 \to S^1 \times \{ \pm 1\}$
defined by
$$
e^{i \theta} \mapsto \wt \kappa_\theta \mapsto
\Big(e^{2i \theta} , \zeta(\theta)\Big).
$$
is the restriction of
$$
z\mapsto \left(z^2, \frac{\sqrt{z^2}'}z\right)
$$
Observe that $\sqrt{z}'$ is the principal value of the square
root of $z$ except when $z \in (-\infty, 0).$
\end{rmk}

Next we discuss the metaplectic phenomena over $B^{+}(\R)$.

\begin{lem}\label{beta trivial on B plus R}
The cocycle $\beta_{\infty}$ is trivial on $B^+(\R) \times SL(2, \R)$
and $SL(2, \R) \times B^+(\R).$
\end{lem}
\begin{proof}
Let
\[
g=\begin{pmatrix} a & b \\ c & d \end{pmatrix}\in SL(2,\R), \qquad g_0=\begin{pmatrix} y^{1/2} & xy^{-1/2} \\ & y^{-1/2} \end{pmatrix}\in B^{+}(\R),
\]
then we have
\[
gg_0=\begin{pmatrix} * & * \\ cy^{1/2} & (cx+d)y^{-1/2} \end{pmatrix}, \qquad g_0g=\begin{pmatrix} * & * \\ cy^{-1/2} & dy^{-1/2} \end{pmatrix}.
\]
In particular, we have
\[
x(g_0)=y^{1/2}>0, \qquad x(g_0g)=x(g_0)x(g), \qquad \operatorname{sgn}x(gg_0)=\operatorname{sgn}x(g).
\]
Hence by definition we have
\[
\beta_{\infty}(g,g_0)=\Big(x(g),x(g_0)\Big)_{\infty}\Big(-x(g)x(g_0),x(gg_0)\Big)_{\infty}=\Big(-x(g),x(gg_0)\Big)_{\infty}=1,
\]
\[
\beta_{\infty}(g_0,g)=\Big(x(g_0),x(g)\Big)_{\infty}\Big(-x(g_0)x(g),x(g_0g)\Big)_{\infty}=\Big(-x(g),x(g)\Big)_{\infty}=1. \qedhere
\]
\end{proof}

\begin{rmk}
By the lemma, we have the injective homomorphism
\[
B^+ (\R)\hookrightarrow\SL(2, \R), \qquad b \mapsto (b,1)
\]
which splits the covering map. We henceforth identify $B^+(\R)$ with its image in
$\SL(2, \R).$
\end{rmk}

Now we are ready to introduce the Iwasawa decomposition over $\SL(2,\R)$.

\begin{lem}
Each element $\wt g$ of $\SL(2, \R)$ has a unique expression
as $\wt g = b \wt \kappa$
with $b \in B^+(\R)$ and $\wt \kappa \in \wt K.$
\end{lem}
\begin{proof}
This follows immediately from
Lemma \ref{beta trivial on B plus R} and
the analogous statement for $SL(2, \R).$
\end{proof}
Based on the above lemma, we are able to define functions
$b: \SL(2, \R) \to B^+(\R)$ and $\theta: SL(2, \R) \to \R/2\pi \Z$ by
$$
\wt g = b( \wt g)\wt \kappa_{\theta(\wt g)}, \qquad
(\wt g \in \SL(2, \R)).$$
\begin{lem}
For $z=x+iy \in \c H,$ define
\[
b_z = \bspm y^{\frac 12 }& xy^{-\frac12}\\ & y^{-\frac12}\espm \in B^+(\R).
\]
Then we have
$$b( \wt g) = b_{\pr(\wt g)\cdot i}.$$
\end{lem}
\begin{proof}
Clearly $\wt g \cdot z := \pr(\wt g)\cdot z$ is an action of $\SL(2, \R)$
on $\c H.$  The stabilizer of $i$ is the preimage of $SO_2(\R),$
which is $\wt K.$  Hence $\wt g\cdot i = b(\wt g)\cdot i.$  And for each $z \in \c H,$ the matrix $b_z$ can
be described as the unique element of $B^+(\R)$ mapping $i$ to $z.$
\end{proof}
In what follows, we shall continue to use the action of $\SL(2, \R)$ on
$\c H$ by $\wt g \cdot z := \pr(\wt g)\cdot z.$
\begin{cor}
For any $\wt g \in \SL(2, \R)$ and $ z \in \c H$ there exists
$\theta(\wt g, z) \in \R/2\pi \Z$ such that
$\wt g \cdot b_z = b_{\wt g \cdot z} \wt \kappa_{\theta(\wt g, z)}.$
\end{cor}
Specifically, $\theta( \wt g, z) = \theta(\wt gb_z),$
where the latter is defined using the Iwasawa decomposition
as above.
It is immediate from the definitions that the function
$\theta: \SL(2, \R) \times \c H \to \R/2\pi \Z$ is a cocycle,
i.e.,
$$
\theta( \wt g_1 \wt g_2 , z) = \theta( \wt g_1, \wt g_2 \cdot z) +
\theta( \wt g_2, z).
$$

Lastly, we discuss the slash operator of $\SL(2,\R)$ upon the functions over $\c H$.
Define
$\wt{\f j}: \SL(2, \R) \times \c H \to S^1$ by
$\wt{\f j}( \wt g, z) = e^{i \theta( \wt g, z)}.$
Clearly, $\wt{\f j}$ is a cocycle, since $\theta$ is.

\begin{lem}
The cocycle $\wt{\f j}( \wt g, z)$ is always a square root of $\f j(\pr(\wt g), z),$ namely
\[
\wt{\f j}(\wt g, z)^2 = \f j( \pr(\wt g), z).
\]
\end{lem}

\begin{proof}
Write $\wt g b_z = b_{\wt g\cdot z}\wt\kappa_{\theta(\wt g, z)}.$
Then $\pr(\wt g) b_z = b_{\wt g\cdot z}\kappa_{2\theta(\wt g, z)}.$
But
$$
\f j( \pr(\wt g), z) =\frac{ \f j( \pr( \wt g) b_z, i)}{\f j(b_z, i)},
$$
and $\f j( b_z, i) =1,$ so we get
$$
\f j( \pr(\wt g), z) =\f j( \pr( \wt g) b_z, i) = e^{i 2\theta(\wt g, z)}
= \wt{\f j}(\wt g, z)^2. \qedhere$$
\end{proof}

Clearly $ \wt{\f j}(\wt g, z)$ is the principal value
of the square root of $\f j( \pr(\wt g), z)$ if and only if
$\theta(\wt g, z) \in ( -\pi/2, \pi/2).$

For $ f: \c H \to \C,$ now we define the slash operator
$$
\left(f\big|^\sim \wt g \right)(z) = \wt{\f j}(\wt g, z)^{-1} f( g \cdot z) = \wt{\f j}(\wt g, z)^{-1} f( \pr(\wt g) \cdot z) \qquad (\wt g \in \SL(2, \R)).
$$

\begin{lem}The slash operator
$\big|^\sim$ gives a well-defined right action of $\SL(2, \R)$
on the space of all functions $\c H \to \C.$
\end{lem}
\begin{proof}
This follows immediately from the fact that $\wt{\f j}$ is a cocycle.
\end{proof}

\begin{lem}
Let $f: \c H \to \C$ and $g \in SL(2, \R)$. Then we have
\begin{equation}\label{eq: Maa slash and tilde slash}
\left(f\big|^\sim (g, 1) \right)
= \left( f\ms g \right) .
\end{equation}
\end{lem}

\begin{proof}
Clearly $ \wt{\f j}((g,1), z)$ is the principal value
of the square root of $\f j( g, z)$ if and only if
$\theta((g,1), z) \in ( -\pi/2, \pi/2),$ which in turn
is equivalent to $\zeta(\theta((g,1),z)) = 1$.

Now in $\SL(2,\R)$ we have
\[
(g,1)b_z=(g,1)(b_z,1)=(gb_z,1),
\]
so by definition $\zeta(\theta((g,1),z)) = 1$. This confirms
that $ \wt{\f j}((g,1), z)$ is the principal value
of the square root of $\f j( g, z)$. Hence
\[
\left(f\big|^\sim (g, 1) \right)
=\wt {\f j}((g,1),z)f(g\cdot z)
=\f j(g,z)^{-\frac{1}{2}}f(g\cdot z)
= \left( f\ms g \right). \qedhere
\]
\end{proof}

As an example, let $\chi\pmod{N}$ be an even Dirichlet character and consider the classical theta function $\theta_{\chi}$. Define
$$
\theta_{\chi}^{\Maa}( x+iy) = y^{\frac 14} \theta_\chi(x+iy).
$$
Then for $\gamma\in\Gamma_0(4N^2)$ we have
\[
\theta_{\chi}^{\Maa}(\gamma z)=\chi(d)\chi_c(d)\varepsilon_d^{-1}
\f j(\gamma,z)^{1/2}\theta_{\chi}^{\Maa}(z).
\]
Hence
\[
\left(\theta_{\chi}^{\Maa} \big|^\sim (\gamma,1)\right)(z)=
\left(\theta_{\chi}^{\Maa} \ms \gamma\right)(z)=\chi(d)\chi_c(d)\varepsilon_d^{-1}
\theta_{\chi}^{\Maa}(z).
\]

\section{Local Weil Representations}

\subsection{Local Weil Representation}
\label{ss: local weil rep}
The Bruhat-Schwartz space of $\Q_v$ will be denoted
$\c S(\Q_v).$  It is the Schwartz space when $v = \infty$
and the space of all locally constant compactly supported
functions when $v$ is a prime.
Following~\cite{Gelbart-PS} we consider the family of
representations
$r^{\psi_v}$
of $\SL(2, \Q_v)$
on $\c S(\Q_v),$
indexed by the nontrivial characters $\psi_v$ of $\Q_v,$
and defined by
$$
\begin{aligned}
\left[r^{\psi_v}\left(\bpm 0&1 \\ -1& 0 \epm,1\right) .
\vph\right](x) &= \gm(\psi_v) \widehat \varphi(x) \\
\left[r^{\psi_v}\left(\bpm 1& b \\ &1 \epm ,1\right). \vph\right](x)
&= {\psi_v}(bx^2) \vph(x)\\
\left[r^{\psi_v}\left(\bpm a&  \\ &a^{-1} \epm,1\right) . \vph\right](x)
&= |a|^{\frac 12} \frac{\gm( {\psi_v})}{\gm({\psi_{v,a}})} \vph(ax)\\
[r^{\psi_v}(I_2,\zeta)\vph]&=\zeta \cdot \vph,
\end{aligned}
$$
where the Fourier transform is given by
\[
\widehat{\varphi}(x)=\alpha(\psi_v)\int_{\Q_v}\varphi(y)\psi_v(2xy)\,dy,
\]
$dy$ is the standard Haar measure over $\Q_v$, $\alpha(\psi_v)$ is the normalization factor such that $\widehat{\widehat{\varphi}}(x)=\varphi(-x)$, and \footnote{Note that the formula for $\gamma(\psi_{v,a})$ on Page 36 of~\cite{Gelbart} contains a typo.}
$$
\gm(\psi_{v,a}) = \begin{cases}
e^{\frac{a}{|a|}\frac{\pi i}{4}}, & \text{if $v = \infty$ and $\psi_{\infty}(x)=e^{2\pi ix}$;} \\
\lim_{m\to -\infty}
\alpha(\psi_{v,a})\int_{p^m \Z_p} \psi_v( ay^2) \, dy , &
\text{if $v$ is a prime.}
\end{cases}
$$

\begin{rmk}
The constant $\gm(\psi_{v,a})$ is an eighth root of
one.  This is obvious when $v = \infty$ and a result
of Weil otherwise. (Cf.~\cite[Page 36]{Gelbart}.)
\end{rmk}

Now we would like to explicitly describe the local Weil representation with respect to the additive character
\[
e_v(x)=\begin{cases}
e^{2\pi i x},                      & \text{if $v=\infty$;} \\
\displaystyle{e^{-2\pi i\{x\}_p}}, & \text{if $v=p$,}
\end{cases}
\]
where, for every prime $p$, we denote by
\[
\{\cdot\}_p: \Q_p\to\Q, \qquad \sum_{n=-N}^{\infty}a_np^n\mapsto\sum_{n=-N}^{-1}a_np^n
\]
the ``$p$-adic fractional part'' of $\Q_p$.

\begin{prop}\label{prop:WeilRepresentationIntegralNormalization}
Let $a\in\Q_v^{\times}$. Then we have
\[
\alpha(e_{v,a})=|2a|_v^{1/2}.
\]
\end{prop}

\begin{proof}
If $v$ is a finite place, say at $p$, then as the test function we take $\varphi(x)=\1_{\Z_p}(x)$. By definition we have
\[
\widehat{\varphi}(x)=\alpha(e_{p,a})\int_{\Z_p}e_p(2axy)\,dy=\alpha(\psi_{p,a})\1_{(2a)^{-1}\Z_p}(x),
\]
so
\[
\widehat{\widehat{\varphi}}(x)=\alpha(e_{p,a})^2\int_{(2a)^{-1}\Z_p}e_p(2axy)\,dy=\frac{\alpha(e_{p,a})^2}{|2a|_p}\1_{\Z_p}(x).
\]
Hence by definition we have $\alpha(e_{p,a})=|2a|_p^{1/2}$.

Now consider the case $v=\infty$. As the test function we take $\varphi(x)=e^{-\pi x^2}$, then
\[
\widehat{\varphi}(x)=\alpha(e_{\infty,a})\int_{-\infty}^{\infty}e^{-\pi y^2+4\pi aixy}\,dy=\alpha(e_{\infty,a})e^{-4\pi a^2x^2},
\]
\[
\widehat{\widehat{\varphi}}(x)=\alpha(e_{\infty,a})^2\int_{-\infty}^{\infty}e^{-4\pi a^2y^2+4a\pi ixy}\,dy=\frac{\alpha(e_{\infty,a})^2}{|2a|}e^{-\pi x^2}.
\]
Hence again we have $\alpha(e_{\infty,a})=|2a|^{1/2}$.
\end{proof}
Next we evaluate $\gamma(e_{p,a}).$
\begin{prop}\label{prop:computation of gammas}
Let $a\in\Z_p\setminus\{0\}$, say with the decomposition $a=\alpha p^r$ for some $r\in\Z$ and $\alpha\in\Z_p^{\times}$. Then we have
\[
\gamma(e_{p,a})=\begin{cases}
\displaystyle{\frac{1+i}{\sqrt{2}}\varepsilon_{-\alpha}^{-1}\left(\frac{2}{-\alpha}\right)^r},           & \text{if $p=2$;}                              \\[2ex]
\displaystyle{ \left(\frac{-\alpha}{p}\right)}, & \text{if $p\equiv 1\pmod{4}$ and $2\nmid r$;} \\[2ex]
\displaystyle{i\left(\frac{-\alpha}{p}\right)}, & \text{if $p\equiv 3\pmod{4}$ and $2\nmid r$;} \\[2ex]
1,                                              & \text{if $2\nmid p$, $2|r$,}
\end{cases}
\]
where, for the factor $\varepsilon_{-\alpha}$ and for the Kronecker symbols involving $\alpha$, we adopt the convention they are evaluated with respect to some $\alpha^{*}\in\Z$ such that $|\alpha-\alpha^{*}|_p$ is sufficiently small.
\end{prop}

\begin{proof}
We have, for $m\gg 1$, that
\begin{align*}
\gamma(e_{p,a})
=& \alpha(e_{p,a})\int_{p^{-m}\Z_p}e_p(ay^2)\,dy=|2a|_p^{1/2}\int_{p^{-m}\Z_p}e(-\{ay^2\}_p)\,dy=|2a|_p^{1/2}p^m\int_{\Z_p}e\left(-\left\{\frac{ay^2}{p^{2m}}\right\}_p\right)\,dy \\
=& |2|_p^{1/2}p^{m-r/2}\int_{\Z_p}e\left(-\left\{\frac{\alpha y^2}{p^{2m-r}}\right\}_p\right)\,dy=|2|_p^{1/2}p^{r/2-m}\sum_{y\in\Z_p/p^{2m-r}\Z_p}e\left(-\frac{\alpha y^2}{p^{2m-r}}\right) \\
=& |2|_p^{1/2}p^{r/2-m}\sum_{y\in\Z/p^{2m-r}\Z}e\left(-\frac{\alpha y^2}{p^{2m-r}}\right).
\end{align*}
To evaluate the inner quadratic Gauss sum, we quote the following famous result of Gauss
\[
\sum_{n=0}^{c-1}e\left(\frac{an^2}{c}\right)=\begin{cases}
\displaystyle{\varepsilon_c\left(\frac{a}{c}\right)\sqrt{c}},      & \text{if $2\nmid c$;} \\[1ex]
\displaystyle{\varepsilon_a^{-1}(1+i)\left(\frac{c}{a}\right)\sqrt{c}}, & \text{if $a$ is odd, $4|c$;} \\[1ex]
0,                                                              & \text{if $c\equiv 2\pmod{4}$.}
\end{cases}
\]

We start with the case that $p=2$. Since we have assumed $m\gg 1$, the quadratic Gauss sum becomes
\[
\sum_{y\in\Z/2^{2m-r}\Z}e\left(-\frac{\alpha y^2}{2^{2m-r}}\right)=\varepsilon_{-\alpha}^{-1}(1+i)\left(\frac{2^{2m-r}}{-\alpha}\right)2^{m-r/2}=\varepsilon_{-\alpha}^{-1}(1+i)\left(\frac{2}{-\alpha}\right)^r2^{m-r/2}.
\]
Hence we have
\[
\gamma(e_{2,a})=\frac{1+i}{\sqrt{2}}\varepsilon_{-\alpha}^{-1}\left(\frac{2}{-\alpha}\right)^r.
\]

Now we assume that $p\neq 2$, then the above result on quadratic Gauss sums shows that
\[
\gamma(e_{p,a})=\varepsilon_{p^{2m-r}}\left(\frac{-\alpha}{p^{2m-r}}\right)=\varepsilon_{p^{2m-r}}\left(\frac{-\alpha}{p}\right)^r.
\]
Since $p^2\equiv 1\pmod{4}$, we have
\[
\varepsilon_{p^{2m-r}}=\varepsilon_{p^r}=\begin{cases}
1, & \text{if $p\equiv 1\pmod{4}$ or $2|r$;} \\
i, & \text{if $p\equiv 3\pmod{4}$ and $2\nmid r$.}
\end{cases}
\]
Hence we have
\[
\gamma(e_{p,a})=\begin{cases}
\displaystyle{ \left(\frac{-\alpha}{p}\right)}, & \text{if $p\equiv 1\pmod{4}$ and $2\nmid r$;} \\[1ex]
\displaystyle{i\left(\frac{-\alpha}{p}\right)}, & \text{if $p\equiv 3\pmod{4}$ and $2\nmid r$;} \\[1ex]
1,                                              & \text{if $2|r$.}
\end{cases} \qedhere
\]
\end{proof}

\subsection{The Real Weil Representation}

In this section, we consider the real Weil representation of $\SL(2,\R)$.

\begin{lem}
Let $\phi_\infty^0(x)=e^{-2\pi x^2}$. Then
$$r^{e_\infty}(\wt \kappa_\theta) \phi_\infty^0
= e^{-i \theta} \phi_\infty^0, \qquad
(\forall \theta \in \R)$$
\end{lem}

\begin{proof}
We recall that
\[
\wt\kappa_{\theta}=(\kappa_{2\theta},\zeta(\theta))=\left(\begin{pmatrix} \cos 2\theta & -\sin 2\theta \\ \sin 2\theta & \cos 2\theta \end{pmatrix},\zeta(\theta)\right).
\]
It is straightforward to verify that the lemma is valid if $\theta$ is of form $n\pi/2$ for some $n\in\Z$, so we may assume henceforth that $\sin 2\theta\neq 0$.

By direct computations we have
\begin{multline*}
\left(\begin{pmatrix} 1 & \cos 2\theta/\sin 2\theta \\  & 1 \end{pmatrix},1\right)
\left(\begin{pmatrix} -1/\sin 2\theta & \\ & -\sin 2\theta \end{pmatrix},1\right)
\left(\begin{pmatrix}  & 1 \\ -1 & \end{pmatrix},1\right)
\left(\begin{pmatrix} 1 & \cos 2\theta/\sin 2\theta \\ & 1 \end{pmatrix},1\right) \\
% =&
% \left(\begin{pmatrix} -1/\sin 2\theta & -\cos 2\theta \\  & -\sin 2\theta \end{pmatrix},1\right)
% \left(\begin{pmatrix}  & 1 \\ -1 & -\cos 2\theta/\sin 2\theta \end{pmatrix},1\right) \\
=\left(\begin{pmatrix} \cos 2\theta & -\sin 2\theta \\ \sin 2\theta & \cos 2\theta \end{pmatrix},-\varepsilon\right),
\end{multline*}
where for simplicity we write $\varepsilon=\sgn\sin 2\theta$. Write
\begin{multline*}
f(x)=r^{e_{\infty}}\left(\begin{pmatrix} 1 & \cos 2\theta/\sin 2\theta \\ & 1 \end{pmatrix}\right)\phi_{\infty}^0(x) \\
=e_{\infty}\left(\frac{\cos 2\theta}{\sin 2\theta}x^2\right)\phi_{\infty}^0(x)=\exp\Big(-2\pi x^2(1-i\cot 2\theta)\Big),
\end{multline*}
then
\begin{align*}
r^{e_{\infty}}(\wt \kappa_{\theta})\phi_{\infty}^0(x)
=& -\varepsilon\zeta(\theta)r^{e_{\infty}}\left(\begin{pmatrix} 1 & \cos 2\theta/\sin 2\theta \\  & 1 \end{pmatrix}\right)r^{e_{\infty}}\left(\begin{pmatrix} -1/\sin 2\theta & 0 \\ 0 & -\sin 2\theta \end{pmatrix}\right) \\
&\cdot r^{e_{\infty}}\left(\begin{pmatrix}  & 1 \\ -1 & \end{pmatrix}\right)r^{e_{\infty}}\left(\begin{pmatrix} 1 & \cos 2\theta/\sin 2\theta \\ & 1 \end{pmatrix}\right)\phi_{\infty}^0(x) \\
=& -\varepsilon\zeta(\theta)r^{e_{\infty}}\left(\begin{pmatrix} 1 & \cos 2\theta/\sin 2\theta \\  & 1 \end{pmatrix}\begin{pmatrix} -1/\sin 2\theta & 0 \\ 0 & -\sin 2\theta \end{pmatrix}\begin{pmatrix}  & 1 \\ -1 & \end{pmatrix}\right)f(x) \\
=& -\frac{\varepsilon\zeta(\theta)}{|\sin 2\theta|^{1/2}}e_{\infty}(x^2\cot 2\theta)\frac{\gamma(e_{\infty})^2}{\gamma(e_{\infty,-1/\sin 2\theta})}\widehat{f}\left(-\frac{x}{\sin 2\theta}\right) \\
=& -\frac{\varepsilon\zeta(\theta)}{|\sin 2\theta|^{\frac{1}{2}}}e^{(2+\varepsilon)\frac{\pi i}{4}+2\pi ix^2\cot 2\theta}\widehat{f}\left(-\frac{x}{\sin 2\theta}\right) \\
=& \frac{\varepsilon\zeta(\theta)}{|\sin 2\theta|^{\frac{1}{2}}}e^{-\frac{\varepsilon\pi i}{4}+2\pi ix^2\cot 2\theta}\widehat{f}\left(-\frac{x}{\sin 2\theta}\right).
\end{align*}

Now recall that, if $\varphi_z(x)=e^{-2\pi zx^2}$ for some $z\in\C$ with $\Re z>0$, then
\[
\widehat{\varphi}_z(x)=\frac{1}{\sqrt{z}}e^{-\frac{2\pi x^2}{z}}.
\]
In our case, we have
\[
z=1+i\cot 2\theta=\frac{\sin 2\theta-i\cos 2\theta}{\sin 2\theta}=\frac{1}{|\sin 2\theta|}e^{2i\theta-\frac{\varepsilon\pi i}{2}},
\]
and, according to our convention,
\[
\frac{1}{z}=(\sin 2\theta)^2(1+i\cot 2\theta), \qquad \sqrt{z}=\frac{\zeta(\theta)}{\sqrt{|\sin 2\theta|}}e^{i\theta-\frac{\varepsilon\pi i}{4}},
\]
so
\[
\widehat{f}(x)=\zeta(\theta)|\sin 2\theta|^{1/2}e^{\frac{\varepsilon\pi i}{4}-i\theta}e^{-2\pi x^2(\sin 2\theta)^2(1+i\cot 2\theta)}.
\]
Hence
\begin{align*}
r^{e_{\infty}}(\wt \kappa_{\theta})\phi_{\infty}^0(x)
=& \frac{\zeta(\theta)}{|\sin 2\theta|^{\frac{1}{2}}}e^{-\frac{\varepsilon\pi i}{4}+2\pi ix^2\cot 2\theta}\widehat{f}\left(-\frac{x}{\sin 2\theta}\right) \\
=& e^{-i\theta}e^{-2\pi x^2}=e^{-i\theta}\varphi_{\infty}^0(x). \qedhere
\end{align*}
\end{proof}

\begin{comment}
\begin{rmk}
Let $\varphi_{\infty}^1(x)=xe^{-2\pi x^2}$, then the same argument can be applied to show that
\[
r^{e_{\infty}}(\wt\kappa_{\theta})\varphi_{\infty}^1=e^{-3i\theta}\varphi_{\infty}^1,
\]
as long as we observe the Fourier transformation formula for $\varphi_z(x)=xe^{-2\pi zx^2}$
\[
\widehat{\varphi}_z(x)=\frac{i}{z\sqrt{z}}xe^{-2\pi x^2/z}.
\]
\end{rmk}
\end{comment}

\subsection{The Nonarchimedean Weil Representation. I}

We denote  the characteristic function
of $\Z_p$ by $\phi_p^\circ.$
\begin{lem}\label{K p fixes phi p circ}
If $p>2$ and $a \in \Z_p^\times,$ then
$r^{e_{p,a}}(SL(2, \Z_p))$ fixes $\phi_p^\circ.$
\end{lem}
\begin{proof}
By lemma \ref{lem: generators for K p}, it
suffices to check the assertion on the elements of the set
\eqref{generators}, which is straightforward.
\end{proof}

Take $\mu: \Z_p^\times \to \C^\times$ a nontrivial character,
and define $\phi_p^\mu = \mu \cdot \1_{\Z_p^\times}: \Q_p^\times
\to \C.$

\begin{prop}\label{prop:action of unscaled flip}
Let $p>2$, and let
\[
f=\min\{m\geq 1 \mid \mu(1+ p^f \Z_p)=1\}.
\]
Then
$$r^{e_p}\bpm 0&1\\ -1& 0 \epm.\phi_p^\mu (y)
= \frac{\mu^{-1}(2)\tau(\mu)}{p^f}\phi_p^{\mu^{-1}}
\left(yp^f\right),$$
where
$$
\tau(\mu) =  \sum_{{i = 1}\atop (p,i)=1}^{p^{f}-1}
\mu(i) e_p\left(\frac{i}{p^f}\right).$$
\end{prop}
\begin{proof}
One may decompose $\phi_p^\mu$ as a linear
combination of characteristic functions:
$$
\phi_p^\mu = \sum_{{i = 1}\atop (p,i)=1}^{p^{f}-1}
\mu(i) \1_{i+ p^{f}\Z_p},
$$
then
$$
r^{e_p}\bpm 0&1\\ -1& 0 \epm.\phi_p^\mu (y)
=
\frac1{p^f}
\sum_{{i = 1}\atop (p,i)=1}^{p^{f}-1}
\mu(i) e_p(2iy) \1_{p^{-f}\Z_p}(y)
=\frac{\mu^{-1}(2)}{p^f}
\sum_{{i = 1}\atop (p,i)=1}^{p^{f}-1}
\mu(i) e_p(iy) \1_{p^{-f}\Z_p}(y)
,
$$
where we have applied our previous results that
\[
\gamma(e_p)=1, \qquad \alpha(e_p)=1.
\]
Obviously the right hand side vanishes
if $y \notin p^{-f}\Z_p.$  Moreover,
if $y \notin p^{-f} \Z_p^\times,$ then $i \mapsto e_p(iy)$ is constant
on $1+p^{f-1}\Z,$ which causes the sum against $\mu$ to
vanish.  Thus the support of
$r^{e_p}\bpm 0&1\\ -1& 0 \epm.\phi_p^\mu$ is
precisely $p^{-f} \Z_p^\times.$
Further,
a change of variables in $i$ shows that
$$
r^{e_p}\bpm 0&1\\ -1& 0 \epm.\phi_p^\mu (yy')
=\mu(y')^{-1}
r^{e_p}\bpm 0&1\\ -1& 0 \epm.\phi_p^\mu (y).
$$
It follows that the function
\[
y \mapsto r^{e_p}\bpm 0&1\\ -1& 0 \epm.\phi_p^\mu \left(\frac y{p^f}\right)
\]
is a scalar multiple of $\phi_p^{\mu^{-1}},$ and the
scalar is easily seen to be the value at $\mu^{-1}(2)\tau(\mu)/p^f$.
\end{proof}
\begin{cor}\label{action of flip}
With notation as before, we have
$$r^{e_p}\left(\bpm 0&p^{-f}\\ -p^{f}& 0 \epm\right).\phi_p^\mu
= \left(\frac{-1}{p}\right)^f\frac{\tau(\mu)\mu^{-1}(2)}{p^{f/2}\gm(e_{p,p^{-f}})} \phi_p^{\mu^{-1}}.$$
\end{cor}
\begin{proof}
By definition we have
\[
\beta_p\left(\bpm p^{-f}&0 \\ 0 & p^f \epm ,\; \bpm 0& 1 \\ -1& 0 \epm
\right)_p
= (p^f, -1)_p(p^f, -p^f)_p
= (p^f, p^f)_p = \left(\frac{-1}{p}\right)^f,
\]
so in $\SL(2,\Q_p)$ we have
\[
\left(\bpm p^{-f}&0 \\ 0 & p^f \epm ,1\right)\left(\bpm 0& 1 \\ -1& 0 \epm,1 \right)
=
\left(\bpm 0 & p^{-f}\\ -p^f & 0 \epm ,\left(\frac{-1}{p}\right)^f \right).
\]
Hence it follows immediately from Proposition \ref{prop:action of unscaled flip} and the definition
of $r^{e_p}$ on diagonal elements that
\begin{align*}
r^{e_p}\left(\bpm 0&p^{-f}\\ -p^{f}& 0 \epm\right).\phi_p^\mu
=& \left(\frac{-1}{p}\right)^fr^{e_p}\left(\bpm p^{-f}&0 \\ 0 & p^f \epm\right)
r^{e_p}\left( \bpm 0& 1 \\ -1& 0 \epm \right)\phi_p^\mu \\
=& \left(\frac{-1}{p}\right)^f\frac{\tau(\mu)\mu^{-1}(2)}{p^{f/2}\gm(e_{p,p^{-f}})}
\phi_p^{\mu^{-1}}. \qedhere
\end{align*}
\end{proof}
Note that $\mu$ factors through $(\Z/p^f \Z)^\times$ and that
$\tau(\mu)$ is the Gauss sum of
the (primitive) Dirichlet character mod $p^f$ which it induces. Therefore, $\frac{\mu^{-1}(2)\tau(\mu)}{\gm(e_{p,p^{-f}})p^{f/2}}$ is a root of unity.

\subsection{The Nonarchimedean Weil Representation. II}\label{sec: the subspace rp}

Let $p$ be a prime, $M\geq 1$, and
$$
K_p^{(M)}:= \bpm 1 & \\ & M \epm SL(2, \Z_p) \bpm 1& \\ & M^{-1} \epm.
$$
Note that $K_p^{(M)}$ depends only on the
$p$-adic valuation of $M.$ By Lemma~\ref{lem: generators for K p}, the group $K_p^{(M)}$ is generated by
\[
\left\{\begin{pmatrix} & 1/M \\ -M & \end{pmatrix}\right\}\bigcup\left\{\begin{pmatrix} 1 & x/M \\ & 1 \end{pmatrix} \mid x\in\Z_p\right\}.
\]
In addition, let $\wt K_p^{(M)}$ denote the
preimage of $K_p^{(M)}$ in $\SL(2, \Q_p).$

  In this section, for each prime
$p$ and for suitable values of $M,$ we study a finite dimensional
subspace of $\c S(\Q_p)$ which is invariant under the
action of $\wt K_p^{(M)}$.
Specifically, when $p=2$ we consider
$\wt K_p^{(8)};$ for odd $p$
we consider $\wt K_p^{(p)}.$ As our discussions and conclusions change dramatically according to whether $p>3$, $p=3$ or $p=2$. Hence we will consider these three cases separately.

\subsubsection{The case $p>3$}\label{subsec: case of odd p}

In order to work explicitly we introduce some notation
from elementary linear algebra.
If $V$ is a complex vector space of finite
dimension $n$, $B = (\beta_1, \dots, \beta_n)$
is an ordered basis for $V,$ and $v$ is a vector
in $V$ then we write $[v]_B$ for the coordinates
of $v$ relative to $B.$ Thus
$$
[v]_B= \bbm x_1\\ \vdots \\ x_n\ebm
\iff \sum_{i=1}^n x_i \beta_i = v.
$$
By identifying $B$ with the row vector $\bbm \beta_1 & \dots  & \beta_n\ebm,$ we may write this succinctly as
$B \cdot [v]_B = v.$
Similarly, if $T:V \to V$ is a linear operator, then
$[T]_B \in \Mat_{n\times n}(\C)$ is the matrix satisfying
$$
[T]_B [v]_B= [Tv]_B \qquad ( \forall v \in V).
$$
Finally, if $B_1$ and $B_2$ are two ordered bases of the
same space, then $_{B_1}c_{B_2}$ is the change-of-basis
matrix satisfying
$$
_{B_1}c_{B_2} [v]_{B_2} = [v]_{B_1},\qquad ( \forall v \in V).
$$

Now let $p>3$ and consider the $\frac{p+1}{2}$-dimensional vector space
\[
V = \Span\left( \left\{ \phi_p^{\xi^2}\mid \xi: ( \Z/p\Z)^\times \to \C\right\}\cup\{\1_{p\Z_p}\}\right)
\]
For our later discussions, it is convenient to construct two other bases for $V$.

For the first basis of $V$, recall that if $p$ is odd then $x \in \Z_p$ is
a square if and only if it is a square modulo $p.$
Now let
$
\1_{\boxed{i}}
$
denote the characteristic function of
$\{ x \in \Z_p: x^2 \equiv i^2 \pmod{p}\}.$
  For example, we have
$$\1_{\boxed{0}} = \1_{p\Z_p}, \qquad \1_{\boxed{1}} = \1_{1+p \Z_p} + \1_{-1+p\Z_p}.$$
Obviously
\[
B_1 = ( \1_{\boxed 1}, \dots, \1_{\boxed{\frac{p-1}2}}, \1_{\boxed{0}})
\]
forms a basis for $V$.

For the second basis of $V$, write $\mu_{\frac{p-1}2}$ for the set
of $(p-1)/2^{\text{th}}$ roots of unity.
If $\psi$ is any Dirichlet character mod $p$,
then the image of $\psi^2$ is in $\mu_{\frac{p-1}2}.$
We parametrize the set of such characters.
Fix a generator $g$ for $(\Z/p\Z)^\times$, and define the character
\[
\psi_j: (\Z/p\Z)^\times \to \mu_{\frac{p-1}2}, \qquad g\mapsto e\left(\frac{2j}{p-1}\right).
\]
Now, our second ordered basis
is
$$
B_2 = ( \phi_p^{\psi_1}, \dots, \phi_p^{\psi_{\frac{p-3}2}}, \phi_p^\circ,
\1_{\boxed{0}}).
$$
\begin{prop}
The Weil representation
$r_p^{e_p}(\wt K_p^{(p)})$ of the group $\wt K_p^{(p)}$ preserves the vector space $V.$  Moreover, we have
\begin{equation}\label{action of upper triangular element on V}
\left[
\left.
r_p^{e_p}\bpm 1 & \frac ap \\ & 1 \epm \right|_{V}\right]_{B_1}
= \bpm e_p(\frac ap) &&&&\\ & e_p(4 \frac{a}p) &&&\\ &&\ddots && \\
&&&e_p\left( \left( \frac{p-1}{2}\right)^2 \frac ap\right) &\\
&&&&1\epm\qquad( a \in \Z_p),\end{equation}
\begin{equation}\label{action of diagonal element on V}
\left[
\left.
r_p^{e_p}\bpm a & \\ & a^{-1} \epm \right|_{V}\right]_{B_2}
=
\bpm \psi_1(a) &&&&\\ &\ddots &&\\&&\psi_{\frac{p-3}{2}}(a) &&\\
&&&1&\\&&&&1\epm
\qquad ( a \in \Z_p^\times),
\end{equation}
\begin{equation}
\label{action of flip on V}
\left[
\left.
r_p^{e_p}\bpm  &p^{-1} \\ -p&  \epm \right|_{V}\right]_{B_2}
= \frac{1}{\varepsilon_p\sqrt{p}} \bpm
&&\tau(\psi_{\frac{p-3}2})\psi_{\frac{p-3}2}^{-1}(2)&&\\
&\iddots &&&\\
\tau(\psi_{1})\psi_{1}^{-1}(2)&&&&\\
&&&& 1\\
&&&p&
\epm. \end{equation}
\end{prop}
\begin{proof}
The identity \eqref{action of upper triangular element on V}
follows immediately from the definitions of $r_p^{e_p}$
and $\1_{\boxed{i}}.$
The identity \eqref{action of diagonal element on V}
follows immediately from the definitions of $r_p^{e_p}$
and the elements of $B_2,$ along with the fact that
for $a \in \Z_p^\times$ we have
$$|a|_p = \gm(e_p) = \gm( e_{p,a}) = 1.$$
To prove \eqref{action of flip on V}, note
that $\psi_j^{-1} = \psi_{\frac{p-1}2-j}.$  Therefore,
by Corollary \ref{action of flip} we have
$$
r_p^{e_p}\bpm  &p^{-1} \\ -p&  \epm
\phi_p^{\psi_j} =  \left(\frac{-1}{p}\right)\frac{\tau(\psi_j)\psi_j^{-1}(2)}{p^{1/2}\gm(e_{p,p^{-1}})} \phi_p^{\psi_{\frac{p-1}2-j}}.
$$
Further, direct calculation shows that
\[
\gamma(e_{p,p^{-1}})=\frac{1}{\varepsilon_p}\left(\frac{-1}{p}\right)
\]
and that
$$r_p^{e_p}\bpm  &p^{-1} \\ -p&  \epm
\phi_p^{\circ}
= \left(\frac{-1}{p}\right)\frac{\sqrt{p}}{\gm(e_{p,p^{-1}})} \1_{\boxed{0}}
= \frac{\sqrt{p}}{\varepsilon_p} \1_{\boxed{0}},
$$
$$
r_p^{e_p}\bpm  &p^{-1} \\ -p&  \epm\1_{\boxed{0}}
= \left(\frac{-1}{p}\right)\frac{1}{\sqrt{p}\gm(e_{p,p^{-1}})} \phi_p^{\circ}
= \frac{1}{\varepsilon_p\sqrt{p}} \phi_p^{\circ}.
$$
The identity \eqref{action of flip on V} follows. The rest
of the proposition follows from these explicit results,
since the elements studied generate $K_p^{(p)}.$
\end{proof}
\begin{defn}\label{defOfVarRho}
Set
$\varrho_{p,B_1}(\gamma) = \left[
\left.
r_p^{e_p}(\gamma)\right|_{V}\right]_{B_1},$  where $V$ and $B_1$ are
defined as above.
\end{defn}
For later use we record the change of basis
matrix from $B_2$ to $B_1.$  It is given by
\begin{equation}\label{change of basis}
_{B_1}c_{B_2}=
\bpm C_0 & 1 & 0 \\ 0 & 1 & 1 \epm
\end{equation}
where $C_0$ is a $\frac{p-1}2\times \frac{p-3}2$ block
with $i,j$ entry equal to $\psi_j(i).$ The $1$ and $0$ in
the first row of $_{B_1}c_{B_2}$ indicate a column of all ones and a
column of all zeros.  The $0$ below $C_0$ is a row of $0$'s.

\subsubsection{The special case $p=3$}
The case $p=3$ is different from the above case in that the only nontrivial Dirichlet character modulo 3 is an odd quadratic character, so our interest is not in the action of the Weil representation upon the space $V$ constructed above but upon this odd quadratic character. The
key observation is that the function
\[
\Z_3^\times\to\C^\times, \qquad x\mapsto e_3\left( \frac{a}{3}x^2\right)
\]
is constant for every $a\in\Z_3$.

\begin{thm}\label{K 3 3 acts on phi chi 3 by xi 3}
Let ${\chi_3}$ be the unique nontrivial quadratic
Dirichlet character modulo $3.$
Then the span of $\phi_3^{\chi_3}$ is
fixed by  $ \wt K_3^{(3)},$
and
there is a character
$\xi_3: \wt K_3^{(3)} \to \C^\times$ such that
$$
r_3^{e_3}(\wt k).\phi_3^{\chi_3} = \xi_3(\wt k) \cdot \phi_3^{\chi_3}
\qquad ( \forall \wt k \in \wt K_3^{(3)}).
$$
On elementary generators, it is given by
$$
\xi_3\left(
\bpm 1& \frac a3\\ & 1\epm, 1
\right)= e_3\left( \frac a3\right), \qquad ( a \in \Z_3)
$$
$$
\xi_3\left(
\bpm a&\\&a^{-1}\epm, 1
\right)= \chi_3(a), \qquad ( a \in \Z_3^\times)
$$
$$
\xi_3\left(
\bpm &1/3\\ -3 & \epm, 1
\right) = 1, \qquad \xi_3( I_2, -1) = -1.
$$
\end{thm}
\begin{proof}
The fact that $\left(
\bpm 1& \frac a3\\ & 1\epm, 1
\right)$ acts by $e_3( \frac a3)$ follows
from the fact that $x^2 \equiv 1 \pmod{3}$
for all $x \in \Z_3^{\times}.$
To study the effect of
$\left(
\bpm a&\\&a^{-1}\epm, 1
\right),$
one only has to invoke the definition and check that
$\gm(e_3)=\gm(e_{3,a})=1$
by Proposition \eqref{prop:computation of gammas}.
The action
of $\xi_3\left(
\bpm &1/3\\ -3 & \epm, 1
\right) $ is given in Corollary \ref{action of flip} with
\[
\tau(\chi_3)=-i\sqrt{3}, \qquad \chi_3(2)=-1, \qquad \gm(e_3)=1, \qquad \gamma(e_{3,1/3})=-i. \qedhere
\]
\end{proof}
\subsubsection{The case $p=2$}
\begin{thm}\label{thm: k 2 8 acts on ph2 chi 2 by xi 2}
Let ${\chi_2}$ denote the unique nontrivial quadratic
Dirichlet character modulo $4.$
Then the span of $\phi_2^{\chi_2}$ is
fixed by  $ \wt K_2^{(8)},$ and there is a character
$\xi_2: \wt K_2^{(8)} \to \C^\times$ such that
$$
r_2^{e_2}(\wt k).\phi_2^{\chi_2} = \xi_2(\wt k) \cdot \phi_2^{\chi_2}
\qquad ( \forall \wt k \in \wt K_2^{(8)}).
$$
On elementary generators, it is given by
$$
\xi_2\left(
\bpm 1& \frac a8\\ & 1\epm, 1
\right)= e_2\left(\frac a8\right), \qquad ( a \in \Z_2)
$$
$$
\xi_2\left(
\bpm a&\\&a^{-1}\epm, 1
\right)=
-i\varepsilon_{-a}
\chi_2(a)
, \qquad ( a \in \Z_2^\times)
$$
$$
\xi_2\left(
\bpm &1/8\\ -8 & \epm, 1
\right) = -\frac{1+i}{\sqrt{2}}, \qquad
\xi_2( I_2, -1) = -1.
$$
\end{thm}
\begin{proof}
The fact that $\left(
\bpm 1& \frac a8\\ & 1\epm, 1
\right)$ acts by  $e_2(\frac a8)$ follows
from the fact that $x^2 \equiv 1 \pmod{8}$ for
all $x \in \Z_2^{\times}.$

Next, $\left(
\bpm a&\\&a^{-1}\epm, 1
\right)$ acts by
$\frac{\gm(e_2)}{\gm(e_{2,a})}\chi_2(a),$
and it follows from the computation
of $\gm(e_{2,a})$ given earlier that
for $a \in \Z_2^\times,$
this is equal to $-i\varepsilon_{-a}\chi_2(a).$

Lastly, we have
$$
\beta_2\left( \begin{pmatrix} & 1 \\ -1 & \end{pmatrix}, \begin{pmatrix} 8 & \\ & 1/8\end{pmatrix}\right)=1.
$$
Since
\[
\left[r_2^{e_2}\left(\begin{pmatrix} 8 & \\ & 1/8\end{pmatrix}\right)\phi_2^{\chi}\right](x)=|8|_2^{1/2}\frac{\gamma(e_2)}{\gamma(e_{2,8})}\phi_2^{\chi}(8x)=\frac{1}{2\sqrt{2}}\phi_2^{\chi}(8x),
\]
we have
\begin{align*}
\left[r_2^{e_2}\left(\begin{pmatrix} & 1 \\ -1 & \end{pmatrix}\right)\right.&{}\left.r_2^{e_2}\left(\begin{pmatrix} 8 & \\ & 1/8\end{pmatrix}\right)\phi_2^{\chi}\right](x)=\frac{1-i}{4\sqrt{2}}\int_{\Q_2}\phi_2^{\chi}(8y)e_2(2xy)\,dy \\
=& \frac{1-i}{4\sqrt{2}}\int_{\Q_2}\Big(\1_{\frac{1}{8}+\Z_2}(y)+\1_{\frac{5}{8}+\Z_2}(y)-\1_{\frac{3}{8}+\Z_2}(y)-\1_{\frac{7}{8}+\Z_2}(y)\Big)e_2(2xy)\,dy \\
=& \frac{1-i}{4\sqrt{2}}\left(e_2\left(\frac{x}{4}\right)+e_2\left(\frac{5x}{4}\right)-e_2\left(\frac{3x}{4}\right)-e_2\left(\frac{7x}{4}\right)\right)\int_{\Z_2}e_2(2xy)\,dy \\
=& \frac{1-i}{4\sqrt{2}}\1_{2^{-1}\Z_2}(x)\left(e_2\left(\frac{x}{4}\right)+e_2\left(\frac{5x}{4}\right)-e_2\left(\frac{3x}{4}\right)-e_2\left(\frac{7x}{4}\right)\right) \\
=& \frac{1-i}{4\sqrt{2}}\1_{2^{-1}\Z_2}(x)e_2\left(\frac{x}{4}\right)\left(1-e_2\left(\frac{x}{2}\right)\right)\Big(1+e_2(x)\Big) \\
=& \frac{1-i}{2\sqrt{2}}\1_{\Z_2}(x)e_2\left(\frac{x}{4}\right)\left(1-e_2\left(\frac{x}{2}\right)\right) \\
=& \frac{1-i}{\sqrt{2}}\1_{\Z_2^{\times}}(x)e_2\left(\frac{x}{4}\right)=\frac{1-i}{\sqrt{2}}\1_{\Z_2^{\times}}(x)e\left(-\frac{x}{4}\right) \\
=& \frac{(-i)(1-i)}{\sqrt{2}}\Big(\1_{1+4\Z_2}(x)-\1_{3+4\Z_2}(x)\Big)=-\frac{1+i}{\sqrt{2}}\phi_2^{\chi}(x). \qedhere
\end{align*}
\end{proof}

\section{Global Metaplectic Group and Weil Representation}

\subsection{Global Metaplectic Group}
\label{ss: Global Metaplectic Group}

If $g= \{ g_v\}_v, \; h = \{ h_v\}_v \in SL(2, \A),$ then
$\beta_v( g_v, h_v) =1$ for all but finitely many $v$
(see~\cite[Proposition 2.8]{Gelbart}).
Set
$$
\beta( g, h) = \prod_v \beta_v(g_v,h_v).
$$
Here $v$ runs over the places
on $\Q.$  Then $\SL(2, \A)$ is defined as $SL(2, \A) \times \{\pm 1\}$
equipped with the product $$
( g_1, \zeta_1)(g_2 \zeta_2) := \Big(g_1g_2, \beta(g_1, g_2)\zeta_1\zeta_2\Big),
\qquad \Big(g_1, g_2 \in SL(2, \A), \; \zeta_1, \zeta_2 \in \{ \pm 1\}\Big).
$$
For each $v,$ we have the embedding
$i_v: SL(2, \Q_v) \hookrightarrow SL(2, \A).$  The definition
is that for $g_v \in SL(2, \Q_v)$ and
 $w$ a place of $\Q,$ the component $i_v(g_v)_w$ of
 $i_v(g_v)$ at $w$ is $g_v$ if $v=w$ and the identity $I_2$
 otherwise.
Now, for all $w,$ the cocycle $\beta_w$ is trivial
on $\{ I_2 \} \times SL(2, \Q_v)$ and $SL(2, \Q_v) \times \{ I_2\},$
which implies that the restriction of the global cocycle
$\beta$ to the image of $SL(2, \Q_v) \times SL(2, \Q_v)$
in $SL(2, \A)$ is precisely the
local cocycle $\beta_v.$
It follows that $i_v$ extends to an embedding $\wt i_v: \SL(2, \Q_v) \hookrightarrow \SL(2, \A)$ defined by
$$\wt i_v( g, \zeta) = (i_v(g), \zeta), \qquad (g \in SL(2, \Q_v), \;\;
\zeta \in \{ \pm 1\}).$$
We shall also make use of the embedding
\[
i_{\on{diag}}: SL(2, \Q)\hookrightarrow\SL(2, \A), \qquad \gamma\mapsto\Big(\gamma,s_{\A}(\gamma)\Big)
\]
as described in~\cite[Page 23]{Gelbart}, where
\[
s_{\A}=\prod_vs_v.
\]
We also let
$$i_{\fin}(\gm) = \wt i_\infty( \gm^{-1}, 1)i_{\on{diag}}(\gm) \in \SL(2, \A),
\qquad (\gm \in
SL(2, \Q) ).$$
Observe that $i_{\fin}$ is {\it not} a homomorphism.  Rather, it satisfies
\begin{equation}\label{nonhom prop of i fin}
i_{\fin}(\gm_1)
i_{\fin}(\gm_2)
=i_{\fin}(\gm_1\gm_2)\cdot( I_2, \beta_\infty(\gm_2^{-1}, \gm_1^{-1})),
\qquad (\gm_1, \gm_2 \in SL(2, \Q)).
\end{equation}
Indeed,
$i_{\fin}(\gm_1)$ commutes with
$\wt i_{\infty} (\gm_2)$ since either one
or the other of them has the identity matrix
at each place.  Hence
$$\begin{aligned}
i_{\fin}(\gm_1)
i_{\fin}(\gm_2)
=& i_{\fin}(\gm_1)
\wt i_\infty( \gm_2^{-1}, 1)i_{\on{diag}}(\gm_2)
=
\wt i_\infty( \gm_2^{-1}, 1)i_{\fin}(\gm_1)i_{\on{diag}}(\gm_2) \\
=&
\wt i_\infty( \gm_2^{-1}, 1)\wt i_\infty( \gm_1^{-1}, 1)i_{\on{diag}}(\gm_1)i_{\on{diag}}(\gm_2) \\
=&
\wt i_\infty( \gm_2^{-1}\gm_1^{-1},
\beta_\infty( \gm_2^{-1}, \gm_1^{-1}))i_{\on{diag}}(\gm_1\gm_2) \\
=& \wt i_\infty( \gm_2^{-1}\gm_1^{-1},
1)i_{\on{diag}}(\gm_1\gm_2)
(I_2, \beta_\infty( \gm_2^{-1}, \gm_1^{-1})).
\end{aligned}$$
By~\cite[Proposition 2.8]{Gelbart}, the restriction
of $\wt i_p$ to $SL(2, \Z_p)$ is a homomorphism
for $p > 2.$  (Cf.~\cite[Page 19]{Gelbart}.)
It follows that
the inclusion
$$
i_S:\prod_{v\in S} SL(2, \Q_v) \times \prod_{p \notin S} SL(2, \Z_p)
\hookrightarrow
SL(2, \A)
$$
extends to a homomorphism
$$\wt i_S:
\prod_{v\in S} \SL(2, \Q_v) \times \prod_{p \notin S} SL(2, \Z_p)
\hookrightarrow
\SL(2, \A)
$$
for any finite set $S$ of places of $\Q$ which contain $\{2,\infty\}.$
The kernel of this homomorphism
is
$$
\ker \wt i_S=
\left\{ (I_2, \varepsilon_v)_{v \in S} \times (I_2)_{p \notin S}: \prod_{v \in S} \varepsilon_v = 1
\right\}.
$$
Moreover
$$
\SL(2, \A) = \bigcup_{S} \on{im} \wt i_S,
$$
with the union ranging over finite sets $S$ of places of $\Q$ which contain $\{2,\infty\}.$

\subsection{Global Weil Representation}
The adelic Bruhat-Schwartz space
$\c S(\A)$ consists of all finite linear combinations of functions
$\prod_v \vph_v,$ where $\vph_v$ is in $\c S(\Q_v)$ for all $v$ and
$\vph_p =\phi_p^\circ$
is the characteristic function
of $\Z_p$ for all but finitely many primes $p.$

For any finite set $S$ of places of $\Q,$ the injection
$$
\bigotimes_{v\in S} \vph_v \mapsto
\bigotimes_{v\in S} \vph_v\otimes \bigotimes_{p \notin S} \phi_p^\circ,
$$
sends $\bigotimes_{v\in S} \c S(\Q_v)$ to a subspace of $\c S(\A).$ The action of $\SL(2, \Q_v)$ on $\c S(\Q_v)$ induces an
action on $\c S(\A)$ for all $v.$
Moreover, the
action of $\{I_2\}\times \{ \pm 1 \} \subset \SL(2, \Q_v)$
is the same (scalar multiplication) for all $v.$
By Lemma \ref{K p fixes phi p circ},
$SL(2, \Z_p)$ fixes $\phi_p^\circ$ for all
but finitely many $p.$
To be precise, if
\[
\psi(\{x_v\}) = \prod_ve_v(ax_v)
\]
for some $a\in\Q^{\times},$ then $SL(2, \Z_p)$ fixes $\phi_p^\circ$ for all $p >2$ such that
$a \in \Z_p^\times.$

Take $S$ a finite set of places containing $\{2,\infty\},$
then the formula
$$
\left[r_S^{\psi}\left( \wt i_S\left(
(g_v,\zeta_v)_{v \in S}\times (k_p,1)_{p \notin S}\right)
\right)\right].\left[\bigotimes_{v \in S} \vph_v
\otimes \bigotimes_{p\notin S} \phi_p^\circ
\right]
= \left(\prod_{v\in S}\zeta_v\right) \cdot \bigotimes_v [r^{\psi_v}(g_v)].\vph_v
\otimes \bigotimes_{p\notin S} \phi_p^\circ
$$
gives a well-defined action of $\SL(2, \A)$ on $\c S(\A).$

\section{The Adelic Theta Functions}
For any $\vph \in \c S(\A)$ and character $\psi: \Q \bs \A \to \C$, define
$$
\Theta_{\ad}^{\psi}(\vph; \wt g) := \sum_{\xi \in \Q} [r^{\psi}(\wt g).\vph](\xi).
$$
It follows from~\cite[Proposition 2.33]{Gelbart}
that $$\Theta_{\ad}^\psi(\vph; i_{\on{diag}}(\gm) \wt g)
= \Theta_{\ad}^\psi(\vph;  \wt g),
\qquad \left( \forall \vph \in \c S( \A), \; \, \wt g \in \SL(2, \A), \; \; \gm \in SL(2, \Q)\right).$$
The function $\Theta_{\ad}^\psi$ is then
an intertwining map from the representation
$r^\psi$ to the representation
of $\SL(2, \A)$ on
automorphic forms by right translation, namely
\begin{equation}\label{Theta an into op}
\Theta_{\ad}^\psi(\vph;  \wt g_1\wt g_2)
=\Theta_{\ad}^\psi(r^\psi(\wt g_2).\vph;  \wt g_1),
\qquad (\wt g_1, \wt g_2 \in \SL(2, \A), \; \; \vph \in \c S(\A)).
\end{equation}

We now construct an adelic theta function
corresponding to the classical theta functions
$\theta_{\chi}$.  The first step is to define
the corresponding
element of
$\c S(\A).$
Recall that
\[
\phi_\infty^0 (u) = e^{-2\pi u^2}, \qquad \phi_p^\circ = \1_{\Z_p},
\]
and that, if $\mu$ is any character of $\Z_p^\times,$ define
$\phi_p^\mu = \mu \cdot \1_{\Z_p^\times}: \Q_p^\times
\to \C.$

Let $\chi\pmod{M}$ be an even Dirichlet
character. One may write $\chi$ uniquely
$$
\chi(m) = \prod_{p\mid M} \chi_p(m),
$$
where $\chi_p$ is a Dirichlet character
modulo $p^{v_p(M)}$ with $p^{v_p(M)}\|M.$
Now, $(\Z/p^{v_p(M)}\Z)^\times$ is identified
with $\Z_p^\times/(1+ p^{v_p(M)}\Z_p),$ so we
may regard $\chi_p$ as a character
of $\Z_p^\times$ which is trivial on $1+ p^{v_p(M)}\Z_p.$ Now let
$$
\begin{aligned}
\phi_v^\chi & = \begin{cases} \phi_v^\circ , & \text{if $v= \infty$ or $v \nmid M$;} \\
\\
\phi_p^{\chi_p}, & \text{if $v = p \mid M$,}
 \end{cases}\\
\phi^\chi(\{ x_v\}_v ) &= \prod_v \phi^\chi_v(x_v),& (\{ x_v\}_v\in \A=\A_\Q)
\end{aligned}
$$
This is an element of the adelic Bruhat-Schwartz space $\c S(\A).$

\begin{lem}\label{lem: relate maass theta to adelic}
Let $z\in \c H$. Then we have
$$
\Theta_{\ad}^e\left(\phi^\chi;
i_\infty\left(
b_z
\right)
\right) = \theta_{\chi}^{\Maa}(z).
$$
\end{lem}

\begin{proof}
For $y>0$, we write
\[
\phi_y^{\chi}(x)=r^e\left(i_{\infty}\left(\bpm y^{\frac{1}{2}} & \\ & y^{-\frac{1}{2}} \epm\right)\right)\phi^{\chi}(x),
\]
then by definition
\begin{align*}
\Theta_{\ad}^e\left(\phi^\chi;i_\infty\left(b_z\right)\right)
=& \sum_{\xi\in\Q}\left[r^e\left(i_\infty\left(\bpm 1& x \\ 0 & 1 \epm\bpm y^{\frac12} & \\ & y^{-\frac12} \epm\right)\right)\phi^{\chi}\right](\xi) \\
=& \sum_{\xi\in\Q}\left[r^e\left(i_\infty\left(\bpm 1 & x \\ 0 & 1 \epm\phi_y^{\chi}\right)\right)\right](\xi)=\sum_{\xi\in\Q}e_{\infty}(x\xi^2)\phi_y^{\chi}(\xi).
\end{align*}
Now
\begin{align*}
\phi_y^\chi(x)
=& r^{e_{\infty}}\left(\bpm y^{\frac{1}{2}} & \\ & y^{-\frac{1}{2}} \epm\right)\phi_{\infty}^{\chi}(x_{\infty})\cdot\prod_p\phi_p^{\chi}(x_p) \\
=& y^{\frac{1}{4}}\phi_{\infty}^{\chi}(y^{\frac{1}{2}}x_{\infty})\cdot\prod_p\phi_p^{\chi}(x_p),
=y^{\frac{1}{4}}e^{-2\pi yx_{\infty}^2}\prod_p\phi_p^{\chi}(x_p),
\end{align*}
so
$$\Theta_{\ad}^e\left(\phi^\chi;i_\infty\left(b_z\right)\right) =y^{\frac{1}{4}}\sum_{\xi\in\Q}e_{\infty}(x\xi^2)\phi_y^{\chi}(\xi)=y^{\frac{1}{4}}\sum_{\xi\in\Q}e^{2\pi ix\xi^2}e^{-2\pi y\xi^2}\prod_p\phi_p^{\chi}(\xi).
$$For every $\xi\in\Q$, direct computations show that
\[
\prod_p\phi_p^{\chi}(\xi)=\begin{cases}
0,       & \text{if $\xi\not\in\Z$;} \\
\chi(\xi), & \text{if $\xi\in\Z$.}
\end{cases}
\]
Hence we have
\begin{align*}
\Theta_{\ad}^e(\phi^\chi;i_\infty(b_z))
=& y^{\frac{1}{4}}\sum_{n\in\Z}\chi(n)e^{2\pi ixn^2-2\pi yn^2} \\
=& y^{\frac{1}{4}}\sum_{n\in\Z}\chi(n)e^{2\pi in^2(x+iy)}=\theta_\chi^\Maa(x+iy). \qedhere
\end{align*}
\end{proof}

\begin{thm}
Let $e = \prod_v e_v$. Then
$$
\Theta_\ad^{e}\left( \phi^\chi; i_\infty( \wt g_{\infty})\right)
= \left(\theta_\chi^\Maa\big |^\sim \wt g\right)(i),
\qquad (\forall \wt g_{\infty} \in \SL(2, \R)).
$$
\end{thm}

\begin{proof}
Let $\wt g_{\infty}\in\SL(2,\R)$, then we have shown the Iwasawa decomposition
\[
\wt g_{\infty}=\begin{pmatrix} y^{1/2} & xy^{-1/2} \\ & y^{-1/2} \end{pmatrix}\wt\kappa_{\theta}=b_{x+iy}\wt\kappa_{\theta}
\]
for some $x\in\R$, $y>0$ and $\theta\in\R$. In particular, we have
$\wt{\f j}(\wt g_{\infty},i)=e^{i\theta}$.
Hence
\begin{align*}
\Theta_{ad}^e(\phi^{\chi},i_{\infty}(\wt g_{\infty}))
=& \sum_{\xi\in\Q}[r^e(i_{\infty}(\wt g_{\infty}))\phi^{\chi}](\xi)=\sum_{\xi\in\Q}[r^e(i_{\infty}(b_{x+iy}))r^{e_{\infty}}(\wt\kappa_{\theta})\phi^{\chi}](\xi) \\
=& e^{-i\theta}\sum_{\xi\in\Q}[r^e(i_{\infty}(b_{x+iy}))\phi^{\chi}](\xi)=e^{-i\theta}\theta_\chi^\Maa(x+iy) \\
=& \wt{\f j}(\wt g_{\infty},i)^{-1}\theta_\chi^\Maa(\wt g_{\infty}\cdot i)=\left(\theta_\chi^\Maa\big |^\sim\wt g_{\infty}\right)(i). \qedhere
\end{align*}
\end{proof}
\begin{cor} \label{slash maa to r e i fin}
If $\gm \in SL(2, \Q)$ then
$$
\left(\theta_\chi^\Maa\ms \gm\right)(z)
= \Theta_\ad^e \Big( r^e(i_{\fin}(\gm^{-1})).\phi^\chi; i_\infty( b_z)\Big).
$$
\end{cor}
\begin{proof}
By equation \eqref{eq: Maa slash and tilde slash},
$$\left(\theta_\chi^\Maa\ms \gm\right)(z)
=\left(\theta_\chi^\Maa \big|^\sim(\gm,1)\right)(z)
=\Theta_\ad^e\left(\phi^\chi; \wt i_\infty( \gm b_z,1)\right).
$$
But $\Theta_\ad^e$ is invariant on the left
by $i_{\diag}( \gm^{-1}),$ so this is equal to
$$\Theta_\ad^e\left(\phi^\chi; i_\fin(\gm^{-1})\wt i_\infty((b_z, 1))\right)
=\Theta_\ad^e\left(\phi^\chi;\wt i_\infty((b_z, 1)) i_\fin(\gm^{-1})\right).$$
Applying \eqref{Theta an into op} completes the
proof.
\end{proof}

\section{Fourier Coefficients of Classical Theta Functions}

\subsection{The First Twist}\label{sec:MainFirst}

\begin{thm}\label{thm: level 576}
Let $\chi_2$ and $\chi_3$ be the nontrivial
quadratic Dirichlet characters modulo $4$ and $3$ respectively.
Let $\chi = \chi_2 \chi_3$
and let $\theta_\chi$ be the associated classical
theta function of level $576.$
Define $\xi_2: \wt K_2^{(8)} \to \C^\times$
as in Theorem \ref{thm: k 2 8 acts on ph2 chi 2 by xi 2} and $\xi_3: \wt K_3^{(3)} \to \C^\times$ as in Theorem \ref{K 3 3 acts on phi chi 3 by xi 3}, and write
\[
\xi: \Gm^{(24)}\to \C^\times \qquad \gamma\mapsto\xi_2(\gm, 1) \xi_3(\gm,1)s_\A(\gm)\beta_{\infty}(\gm^{-1},\gm).
\]
Then for all $\sg \in \Gm^{(24)}$ we have
$$
\theta_\chi\hs \sg = \xi( \sg^{-1})
\theta_\chi.
$$
\end{thm}
\begin{rmk}
Note that $\xi$ is not a homomorphism.  In fact,
$\xi$ is the composition of a
homomorphism with $i_\fin.$  But as in
\eqref{nonhom prop of i fin}, $i_\fin$ is not
a homomorphism.
\end{rmk}
\begin{proof}
By Corollary \ref{slash maa to r e i fin}
$$
\left(\theta_\chi^\Maa\ms \sg \right) (z)
=\Theta_\ad^e\Big( r^e( i_\fin(\sg^{-1})).\phi^\chi; i_\infty(b_z)\Big).
$$
Using Theorems \ref{K 3 3 acts on phi chi 3 by xi 3}
and \ref{thm: k 2 8 acts on ph2 chi 2 by xi 2}
as well as Lemma \ref{K p fixes phi p circ} and
the definitions of $s_\A, i_\fin$  and $\xi$ yields,
$$\Theta_\ad^e\Big( r^e( i_\fin(\sg^{-1})).\phi^\chi; i_\infty(b_z)\Big)
= \xi(\sg^{-1}) \Theta_\ad^e( \phi^\chi; i_\infty(b_z))
=  \xi(\sg^{-1})\theta_\chi^\Maa (z).
$$
Multiplying by $\Im(z)^{-\frac 14}$ gives the
result.
\end{proof}

\begin{cor}\label{cor:MainFirst}
Let $\f a$ be a cusp for $\Gamma_0(576)$ and $\sigma\in\Gamma^{(24)}$ a scaling matrix for $\f a$. Then we have
\begin{equation}\label{a theta sg = xi sg a theta}
A_{\theta_\chi}( \sg, n) = \xi( \sg^{-1}) A_{\theta_\chi}( I_2, n).
\end{equation}
\end{cor}

\begin{proof}
The existence of a scaling matrix $\sigma\in\Gamma^{(24)}$ is guaranteed by Lemma~\ref{lem:ScalingMatrixConstruction}, and the rest follows immediately from the theorem.
\end{proof}

\begin{thm}\label{thm:MainFirst}
Under the notations of Theorem~\ref{thm: level 576}, let $\f a=u/w$ be a cusp of $\Gamma_0(576)$. Then we have
\begin{align*}
A_{\theta_{\chi}}(\sigma_{\f a}^0,n)
=& \begin{cases}
\displaystyle{\xi_2(\sigma_{\f a}^{-1})\xi_3(\sigma_{\f a}^{-1})s_2(\sigma_{\f a}^{-1})s_3(\sigma_{\f a}^{-1})e\left(-\frac{n^2wr}{24u[24,w]}\right)\chi(m)}, & \text{if $n=m^2\geq 1$,} \\
0, & \text{if otherwise,}
\end{cases}
\end{align*}
where we choose $r,s\in\Z$ such that $24us-wr=(24,w)$, and the scaling matrices $\sigma_{\f a}^0$ and $\sigma_{\f a}$ are as given in~\eqref{eq:ScalingMatrixCommon} and~\eqref{eq:ScalingMatrixSpecial} respectively.
\end{thm}

\begin{proof}
By Corollary~\ref{cor:MainFirst} we have
\[
A_{\theta_\chi}(\sigma_{\f a},n)=\xi_2(\sigma_{\f a}^{-1})\xi_3(\sigma_{\f a}^{-1})s_{\A}(\sigma_{\f a}^{-1})\beta_{\infty}(\sigma_{\f a},\sigma_{\f a}^{-1}) A_{\theta_\chi}(I_2,n).
\]
To begin with, we note that
\[
\beta_{\infty}(\sigma_{\f a},\sigma_{\f a}^{-1})=([24,w],-[24,w])_{\infty}=1.
\]
Further, since $w|576$, we have $s_p(\sigma_{\f a},\sigma_{\f a}^{-1})=1$ for every $p\geq 5$. Hence
\[
A_{\theta_\chi}(\sigma_{\f a},n)=\xi_2(\sigma_{\f a}^{-1})\xi_3(\sigma_{\f a}^{-1})s_2(\sigma_{\f a}^{-1})s_3(\sigma_{\f a}^{-1})A_{\theta_\chi}(I_2,n).
\]
Hence the theorem follows immediately as we recall that
\[
A_{\theta_{\chi}}(\sigma_{\f a}^0,n)=e\left(-\frac{nwr}{24u[24,w]}\right)A_{\theta_{\chi}}(\sigma_{\f a},n). \qedhere
\]
\end{proof}

\subsection{The Higher Twists}

Now we consider the higher twists of the classical theta functions. In this section, we let $p\geq 5$.

\begin{thm}
Let $\psi\pmod{p}$ be an even Dirichlet character and $\gm \in \Gm^{(24p)}.$
Then we have
\[
\theta_{\chi\psi}^\Maa\hs \gm(z)
= \xi(\gm^{-1})\Theta_{\ad}^e\left(\Big[
r_p^{e_p}( \gm^{-1},1).\phi_p^{\psi_p}\Big]
\cdot \phi_2^{\chi_2}
\cdot \phi_3^{\chi_3}
\cdot \prod_{v \ne 2,3,p} \phi_v^\circ ; i_{\infty}(b_z)
\right).
\]
where $\xi$ is as defined in Theorem~\ref{thm: level 576}.
\end{thm}
\begin{proof}
We have
\begin{align*}
\theta_{\chi\psi}^\Maa\hs \gm(z)
=& \Theta_{\ad}^e\left(
\phi^{\chi\psi}
; i_\infty( \gm b_z)\right)
=
\Theta_{\ad}^e\left(
\phi^{\chi\psi}
;
\wt i _{\on{diag}}(\gm^{-1})
i_\infty( \gm b_z)\right) \\
=& s_\A(\gm^{-1})\beta_{\infty}(\gamma^{-1},\gamma)
\Theta_{\ad}^e\left(
\phi^{\chi\psi}
;
\prod_p i_p(\gm^{-1},1)
i_\infty(b_z)\right) \\
=& s_\A(\gm^{-1})\beta_{\infty}(\gamma^{-1},\gamma)
\Theta_{\ad}^e\left(
\phi^{\chi\psi}
;
i_\infty(b_z)
\prod_p i_p(\gm^{-1},1)\right) \\
=& s_\A(\gm^{-1})\beta_{\infty}(\gamma^{-1},\gamma)
\Theta_{\ad}^e\left(
\phi_\infty^0 \cdot
\prod_p r_p^{e_p}(\gm^{-1},1).\phi_p^{\chi_p\psi_p}
;
i_\infty(b_z)
\right). \qedhere
\end{align*}
\end{proof}

\begin{thm}\label{thm: level higher}
Fix a generator
for $(\Z/p\Z)^\times$, and adopt the notations from Section
\ref{subsec: case of odd p}. Further, we define $\chi$ and $\xi$ as in Theorem \ref{thm: level 576}.
Then for $\sg \in \Gm^{(24p)},$
$$
A_{\theta_{\chi \psi_j}}(\sg, n)
= 2\xi(\sg^{-1})\chi(n)\cdot {}^t \! e_{i(n)} \cdot
\varrho_{B_1, p}(\sg^{-1}) \cdot {}_{B_1}\!c_{B_2} \cdot e_j,$$
where
$i(n) \in \{ 1,\dots, \frac{p+1}2\}$
is $\frac{p+1}2$ if $p|n$ and otherwise
is the unique element of $\{ 1,\dots, \frac{p-1}2\}$
satisfying $i(n)^2 \equiv n^2 \pmod{p}.$
\end{thm}
\begin{proof}
Arguing as before, we find that
$$
\left(\theta_{\chi \psi_j}^{\Maa} \ms \sg\right) (z)
= \xi(\sg^{-1})\Theta_\ad^e\left([r^{e_p}(\sg^{-1}, 1).\phi_p^{\psi_j}]\cdot \prod_{v\ne p} \phi_v^{\chi}  ; \;\;i_\infty( b_z)\right).
$$
Note that $e_j = [ \phi_p^{\psi_j}]_{B_2}.$
Thus
$$\Big[r^{e_p}(\sg^{-1}, 1).\phi_p^{\psi_j}\Big]_{B_1}
= \left[\left.r^{e_p}(\sg^{-1}, 1)\right|_{V}\right]_{B_1}
=\varrho_{B_1, p}(\sg^{-1}) \cdot {}_{B_1}c_{B_2} \cdot e_j.
$$
Thus
\[
^t e_k\cdot \varrho_{B_1, p}(\sg^{-1}) \cdot {}_{B_1}c_{B_2} \cdot e_j.
\]
is the coefficient of $\1_{\boxed{k}}$ in the
expansion of $[r^{e_p}(\sg^{-1}, 1).\phi_p^{\psi_j}]$
in terms of $B_1.$
For $k \in \{ 0, \dots, \frac{q-1}2\},$ let
$$\theta_{\chi, [k]}^{[p]}(z)
= \Theta_\ad^e \left( \1_{\boxed{k}} \cdot \prod_{v \ne p} \phi^\chi; i_\infty( b_z) \right).
$$
Let $\pmb{\theta}_\chi^{[p]}= \bbm
\theta_{\chi, [1]}^{[p]}, &
\dots, & \theta_{\chi, [\frac{p-1}2]}^{[p]},
& \theta_{\chi, [0]}^{[p]}\ebm.$
Then by linearity of $\Theta_\ad^e$ in the
first argument, one obtains,
$$\left(\theta_{\chi \psi_j}^{\Maa} \ms \sg\right)
= \xi(\sg^{-1})\cdot \pmb{\theta}_\chi^{[p]}\cdot \varrho_{B_1, p}(\sg^{-1}) \cdot {}_{B_1}c_{B_2} \cdot e_j.$$
Now, following the proof of Lemma \ref{lem: relate maass theta to adelic}, one readily checks that
$$
\theta_{\chi, [k]}^{[p]}(z) =\sum_{n=-\infty}^\infty
\1_{\boxed{k}}(n) \chi(n) e^{-2\pi i n^2 z}.
$$
Thus
$$
A_{\theta_{\chi, [k]}^{[p]}}(I_2, n)
= \begin{cases}
2\chi(n), & \text{if $k\equiv i\pmod{n}$,} \\
0,        & \text{if $k\not\equiv i\pmod{n}$.}
\end{cases}
$$
From here, the result follows easily.
\end{proof}

\end{document}